\documentclass[12pt]{amsart}

\usepackage{bbm}

\usepackage{amsmath,graphics}
\usepackage{amsfonts,amssymb, dsfont, upgreek}
\usepackage{amscd}
\usepackage[all]{xy}       
    \SelectTips{cm}{10}     
    \everyxy={<2.5em,0em>:}

\newtheorem{Thm}{Theorem}[section]
\newtheorem{Conj}[Thm]{Conjecture}
\newtheorem{Prop}[Thm]{Proposition}
\newtheorem{Def}[Thm]{Definition}
\newtheorem{Def/Thm}[Thm]{Definition/Theorem}
\newtheorem{Cor}[Thm]{Corollary}
\newtheorem{Lemma}[Thm]{Lemma}

\theoremstyle{remark}
\newtheorem{Rmk}[Thm]{Remark}

\numberwithin{equation}{section}
\newcommand{\ti }{\times}
\newcommand{\ot }{\otimes}
\newcommand{\ra }{\rightarrow}

\newcommand{\lra }{\longrightarrow}

\newcommand{\Hom }{{\mathrm{Hom}}}

\newcommand{\Spec}{{\mathrm{Spec}}}

\newcommand{\Ker}{{\mathrm{Ker}}}

\newcommand{\Pic}{{\mathrm{Pic}}}

\newcommand{\rank }{{\mathrm{rank}}}

\newcommand{\cO}{{\mathcal{O}}}

\newcommand{\cL}{{\mathcal{L}}}
\newcommand{\cE}{{\mathcal{E}}}

\newcommand{\cH}{{\mathcal{H}}}
\newcommand{\cP}{{\mathcal{P}}}

\newcommand{\cC}{{\mathcal{C}}}

\newcommand{\cX}{{\mathcal{X}}}

\newcommand{\G}{{G}}

\newcommand{\NN}{{\mathbb N}}

\newcommand{\PP }{{\mathbb P}}

\newcommand{\QQ }{{\mathbb Q}}
\newcommand{\CC }{{\mathbb C}}
\newcommand{\ZZ }{{\mathbb Z}}

\newcommand{\ke }{{\varepsilon }}

\newcommand{\ka }{{\alpha}}

\newcommand{\X}{\mathfrak{X}}

\newcommand{\WmodG}{W/\!\!/\G}

\newcommand{\frC}{\mathfrak{C}}

\newcommand{\fY}{\mathfrak{Y}}

\newcommand{\WtG}{W/\!\!/_{\!\theta}\G}
\newcommand{\uX}{\underline{X}}
\newcommand{\ucE}{\underline{\cE}}
\newcommand{\uC}{\underline{C}}
\newcommand{\ux}{\underline{x}}

\newcommand{\tw}{\mathrm{tw}}
\newcommand{\fB}{\mathfrak{B}}
\newcommand{\fQmap}{\mathfrak{Qmap}}

\newcommand{\reg}{\mathrm{reg}}
\newcommand{\Qke}{Q_{g,k}^{\ke}(X, \beta )}
\newcommand{\Qked}{Q_{g,k}^{\ke}(X, d )}
\newcommand{\Eff}{\mathrm{Eff}}

\newcommand{\bx}{\mathbf{x}}

\newcommand{\QGke}{QG^{\ke}_{0, k, \beta}(X)}
\newcommand{\bIX}{\bar{I}_{\mu}X}
\newcommand{\vir}{\mathrm{vir}}
\newcommand{\T}{T}

\newcommand{\bL}{\mathbb{L}}

\newcommand{\lan}{\langle}
\newcommand{\ran}{\rangle}
\newcommand{\lla}{\langle\!\langle}
\newcommand{\rra}{\rangle\!\rangle}
\newcommand{\one}{{\mathbbm 1}}

\newcommand{\QGs}{QG^{\ke}_{0, [k]\cup\{\star\}, \beta }(X)}
\newcommand{\Qa}{Q^{\ke}_{0, A_1\cup \{\bullet\}, \beta _1 }(X)}
\newcommand{\Qb}{Q^{\ke}_{0, A_2\cup \{\check{\bullet}\}, \beta _2 }(X)}

\newcommand{\HA}{H}

\newcommand{\IX}{I_{\mu}X}

\newcommand{\br}{\mathbf{r}}
\newcommand{\ev}{\tilde{ev}}
\newcommand{\orb}{\mathrm{orb}}

\newcommand{\Giv}{\mathrm{Giv}}

\newcommand{\bIY}{\bar{I}_{\mu}Y}
\newcommand{\IY}{I_{\mu}Y}
\newcommand{\cK}{\mathcal{K}}

\newcommand{\e}{\mathbf{e}}

\newcommand{\Aut}{\mathrm{Aut}}

\newcommand{\fM}{\mathfrak{M}}
\newcommand{\fMtw}{\mathfrak{M}^{\tw}_{g, k}}

\newcommand{\pr}{\mathrm{pr}}

\newcommand{\PPa}{\PP _{a,1}}

\newcommand{\Homa}{\Hom _{\beta}(\PPa, X)}

\newcommand{\Qmaprep}{Q_{\PP _{a,1}}(X, \beta)}

\newcommand{\PPb}{\PP _{\bullet, 1}}

\newcommand{\Qket}{Q^\ke _{0,2+m} (X, \beta)}

\begin{document}
\title{Orbifold Quasimap Theory}

\begin{abstract} We extend to orbifolds the quasimap theory of \cite{CK, CKM}, as well as the genus zero wall-crossing results from \cite{CKg0, BigI}.
As a consequence, we obtain generalizations of orbifold mirror theorems, in particular, of the mirror theorem for toric orbifolds recently proved independently
by Coates, Corti, Iritani, and Tseng \cite{CCIT}.
\end{abstract}

\author{Daewoong Cheong}
\address{School of Mathematics, Korea Institute for Advanced Study,
85 Hoegiro, Dongdaemun-gu, Seoul, 130-722, Korea}
\email{daewoongc@kias.re.kr}
\author{Ionu\c t Ciocan-Fontanine}
\noindent\address{School of Mathematics, University of Minnesota, 206 Church St. SE,
Minneapolis MN, 55455, and\hfill
\newline \indent School of Mathematics, Korea Institute for Advanced Study,
85 Hoegiro, Dongdaemun-gu, Seoul, 130-722, Korea}
\email{ciocan@math.umn.edu}
\author{Bumsig Kim}
\address{School of Mathematics, Korea Institute for Advanced Study,
85 Hoegiro, Dongdaemun-gu, Seoul, 130-722, Korea}
\email{bumsig@kias.re.kr}


\maketitle

\tableofcontents

\section{Introduction} Orbifold Gromov-Witten theory was first introduced by Chen and Ruan \cite{CR}, with an algebraic version due later to Abramovich, Graber, and Vistoli \cite{AGV}.

The theory of $\ke$-stable quasimaps to a large class of GIT quotient targets was developed in \cite{CKM}, generalizing and unifying the earlier works \cite{MOP,CK, MM, Toda}. 
In the appropriate general context of the theory, the  GIT target is a smooth Deligne-Mumford stack (or orbifold), but
to keep the
technicalities at a reasonable level, \cite{CKM} worked under the assumption that the GIT quotient is a smooth variety and delegated the orbifold case to subsequent work. In \S2 of this paper 
we formally establish the foundations of quasimap theory for orbifold GIT targets by combining the formalism of \cite{CKM} with the one developed in \cite{AGV} for
algebraic orbifold Gromov-Witten theory. 

Namely,
consider a triple $(W,\G,\theta)$ consisting of an affine variety $W$ over $\CC$, a reductive complex algebraic group $\G$ acting on $W$, and a character $\theta$ of $\G$.
When there are no strictly semistable points for the linearization induced by $\theta$, we have the  {\it GIT stack quotient} $X:=[W^{ss}(\theta)/\G]=[W^{s}(\theta)/\G]$. 
We construct a family, depending on a stability
parameter $\ke\in \QQ_{>0}\cup\{0+,\infty\}$, of (relative) compactifications of moduli stacks of maps from irreducible twisted marked curves with fixed numerical data to the 
GIT stack quotient $[W^{ss}/\G]$. These compactifications
are themselves modular and
we prove in Theorem \ref{Found_Thm} that they are Deligne-Mumford stacks, proper over the affine quotient. Furthermore, if the singularities of $W$ are at worst lci and the semistable locus $W^{ss}$ is nonsingular, these moduli spaces carry canonical perfect obstruction theories and therefore possess
virtual fundamental classes. 

Once the moduli spaces with the required properties are constructed, 
the descendant orbifold $\ke$-quasimap invariants associated to the triple $(W,\G,\theta)$ are defined as usual via integration against the virtual class of products of tautological psi-classes and of
Chen-Ruan cohomology classes pulled-back from the (rigidified) inertia stack of $X$ via the evaluation maps.
When $\ke>2$, they recover the orbifold Gromov-Witten invariants of the DM stack $X$.

It is natural to seek wall-crossing formulas for the invariants as the stability
parameter $\ke$ varies. For triples giving GIT quotients which are nonsingular varieties, such wall-crossing formulas in genus zero were conjectured in \cite{CKg0} as equalities of generating
series of the invariants after a change of variables. 
The formulas were proved in \cite{CKg0} 
in the presence of an action of a torus $T$ on $W$
such that the fixed loci of the induced $T$-action on $\WmodG$ have good properties. 
In section 3 we provide analogous results for the orbifold case. We describe them informally in this Introduction and refer the reader to Conjecture \ref{Main_Conj} and Theorem \ref{Comp_Thm} for the precise statements. These results may be 
viewed as generalized mirror theorems for orbifolds. 

Recall that the genus zero orbifold Gromov-Witten theory of $X$ is encoded in Givental's Lagrangian cone $\cL_X$ and this cone is completely determined by the big $J$-function 
 $J^\infty(q,t,z)$ of $X$ (the notation reflects that GW theory corresponds to the stability parameter $\ke=\infty$).
This is the generating series for all Gromov-Witten invariants with at most one descendent insertion and arbitrary number of primary insertions. 
It depends on the Novikov variables $q$, the general Chen-Ruan cohomology class $t$, and a formal variable $z$.
The $t$-derivatives of the $J$-function determine the
so-called $S$-operator $S_t^\infty$. Conversely, the string equation says that the $J$-function is obtained from the $S$-operator: $J^\infty(q,t,z)=S^\infty_t(\one_X)$, where $\one_X$ is the fundamental class.
Since the terms of $S$-operator involve invariants with at least two insertions, it has a direct quasimap analogue $S^\ke_t$ for every stability parameter. The wall-crossing formula 
for $S$-operators then reads
 \begin{equation}\label{mirror1} S^\ke_t(\one_X)=S^\infty_{\tau^\ke(t)}(\one_X),\end{equation}
 with  the ``mirror map" $\tau^\ke(t)$ a generating series for primary $\ke$-quasimap invariants with a fundamental class insertion. In particular,
 the mirror map acquires an enumerative interpretation. Theorem \ref{Comp_Thm} proves \eqref{mirror1} under the assumption that there is a $T$-action whose fixed points in $X$ are isolated.

Because some genus zero quasimap moduli spaces require at least two markings, a different construction is needed to extend the $J$-function to other stability parameters. To this end, recall that
the Gromov-Witten big $J$-function has a well-known expression as a generating series of certain localization residues for the natural $\CC^*$-action on {\it graph spaces} - moduli spaces of genus zero stable maps with one parametrized component of the domain curve. The graph spaces exist for any stability parameter $\ke\in \QQ_{>0}\cup\{0+,\infty\}$ and the corresponding localization residues define
the big $J^\ke$-function of the $\ke$-quasimap theory. The wall-crossing/mirror formula for big $J$-functions is then the statement that $J^\ke(q,t,z)$ is on the Lagrangian cone $\cL_X$ 
for all $\ke\geq 0+$; this is the second part of Theorem \ref{Comp_Thm}. For semi-positive targets (see Definition \ref{semi-pos}) this follows from the formula for $S$-operators. Without the semi-positivity condition a separate proof is needed and
our argument requires the additional assumption that the one-dimensional $T$-orbits in $X$ are also isolated.

If one is primarily interested in determining the orbifold Gromov-Witten invariants of $X$, a wall-crossing formula becomes useful when the quasimap side can be explicitly computed. 
While in general the $J^\ke$-functions are equally hard to compute for all stability parameters,
in section 4 of the paper we consider a version of moduli spaces of quasimaps, dubbed `` stacky loop spaces", which will typically allow 
one to find closed formulas
for the {\it small} $I$-function $I(q,z):=J^{0+}(q,0,z)$, that is, for the specialization at $t=0$ of the $J$-function for the asymptotic stability condition $\ke=0+$.
Following \cite{BigI}, we also introduce a new orbifold ``big $I$-function"  $\mathds{I}_X(q,t,z)$. 
It is an explicit modification of the small $I$-function by certain exponential factors and
it depends on a parameter $t$ which runs over the part of $H^*(X)$ generated by Chern classes of line bundles associated to characters of $G$.
As in \cite{BigI}, the new $\mathds{I}$-function can be viewed as arising from a variant with weighted markings of quasimap theory. The second main result of the paper, Theorem \ref{big_I}, states that 
$\mathds{I}_X(q,t,z)$ is on the Lagrangian cone $\cL_X$ whenever the $T$-action has isolated fixed points and isolated one-dimensional orbits.

As an application of the theory developed in the paper, in section 5 we discuss the case of toric DM stacks. These are the quotients $X=[W^{ss}/\G]$ for $W$ a vector space and $\G\cong(\CC^*)^r$
an algebraic torus.
We make Theorem \ref{big_I} completely explicit by calculating the small $I$-function $I(q,z)$ in closed form via stacky loop spaces. 
This gives a closed form for the ``big" $\mathds{I}$-function as well.
As a result, the Mirror Theorem for toric orbifolds, recently proved by different methods in \cite{CCIT}, becomes a special case of Theorem \ref{big_I}, see Corollary \ref{toric-explicit}.
\subsection{Acknowledgments} 
The research of D.C. was partially supported by the NRF grant 2007-0093859.
The research of I.C.-F. was partially supported by the NSF grant DMS-1305004. 
The research of B.K. was partially supported by the NRF grant 2014-001824.
I.C.-F. thanks KIAS for excellent
working conditions during visits when this research was conducted.

\section{The stack of stable quasimaps to an orbifold GIT target}

\subsection{Conventions and notations}

We work over the field $\CC$.  All schemes are  locally of finite type over $\Spec\CC$ unless otherwise stated.
Associated to a DM stack $X$ of finite type over $\CC$, we have the {\it (cyclotomic) inertia stack} $\IX$ and its rigidification $\bIX$ (see \cite{AGV}).
For DM stacks $X, \bIX, ...$, we denote by $\uX, \underline{\bIX}, ...$  their coarse moduli spaces.

\subsection{Quotients}

Let $W$ be an irreducible affine variety with a right action
 of a reductive
algebraic group $\G$. Denote by $\CC _\theta$ the one dimensional $G$ representation space
associated to a fixed character $\theta$ of $\G$.  Denote by $W^{ss}$ (or $W^{ss}(\theta)$) 
the semistable locus, and  by $W^{s}$ (or $W^s(\theta)$) the stable locus with respect to the linearization $L_\theta:=W\times\CC_\theta$ (see e.g., \cite{King} for the definitions).
There are then four quotients with natural morphisms between them, summarized in the diagram 
\begin{equation}\label{quot_diag}
 \xymatrix{     X := [W^{ss}/\G]  \ar@{^{(}->}[r]  \ar[d]  & \X:=[W/\G]   \ar[d] \\
    \underline{X}:= \WtG 
   \ar[r] &  \uX _0:= \mathrm{Spec}\CC[W]^{\G} .    } 
   \end{equation}
   In the top line the brackets denote as usual the stack quotients, and the arrow is the open embedding induced by the inclusion $W^{ss}\subset W$.
The GIT quotient  $ \WtG$ is defined to be  \[  \mathbf{Proj}\oplus _{m=0}^\infty \Gamma (W, L_\theta ^m)^{\G}\]
and the bottom arrow is the obvious projective morphism to the affine quotient $\mathrm{Spec}\CC[W]^{\G} $. (Note also that the affine quotient coincides with $W/\!\!/_{\! 0}G$.)
The stack $[W^{ss}/G]$ will be called the {\em GIT stack quotient} with respect to $\theta$.
The natural vertical morphisms in \eqref{quot_diag} are obtained from the fact that principal $G$-bundles are categorical quotients,
since the morphisms $$W^{ss}(\theta) \ra \WtG,\;\;\;\; W\ra W/\!\!/_{\!0}G$$
commute with the $G$-actions.

In this paper we assume that $$W^{ss}(\theta)=W^{s}(\theta)$$ unless otherwise stated so that the GIT stack quotient $X=[W^{ss}/G]$ is 
a quasi-compact DM-stack. 
The morphism $X\ra \uX $ is the coarse moduli and is a proper morphism, see e.g. \cite{KM}. 
Therefore $X$ is proper over $\uX _0$ since $\uX$ is projective over the affine quotient $\uX _0$.      

\subsection{Stable quasimaps to $X$}\label{Def_Qmaps}
Let $\e$ be the least common multiple 
of  the exponents $|\Aut (\bar{p})|$ of automorphism groups $\Aut (\bar{p})$ of all geometric points $\bar{p}\ra X$ of $X$. 

 \begin{Def}\label{Def_qmap}
  Let $(C, x_1,...,x_k)$ be a $k$-pointed, genus $g$ twisted curve, see \cite[\S 4]{AGV}
 and let $\phi: (C, x_1, ..., x_k)  \ra (\uC , \ux_1, ..., \ux_k )$ be its coarse moduli space. 
 
A representable morphism $[u]$ from $(C, x_1,...,x_k)$
to $\X$ is called a $k$-pointed, genus $g$ {\em quasimap to $X$} (alternatively, a {\em  $\theta$-quasimap to $\X$}) if 
$[u]^{-1}(\X \setminus X)$ is purely zero-dimensional.

The  locus $[u]^{-1}(\X \setminus X)$ is called the {\em base locus} of $[u]$.

The {\em  class} $\beta$ of the quasimap is the group homomorphism
\[ \beta: \Pic \X  \ra \QQ
, \ \ L \mapsto \deg ([u]^*(L)) .\]
The rational number $\beta(L_\theta)=\deg ([u]^*(L_\theta))$ is called the {\em degree} of the quasimap $[u]$.

A  quasimap $((C, x_1,.., x_k), [u])$ is called {\em prestable} if the base locus contains
neither marked gerbes nor nodal gerbes of $(C, x_1,..., x_k)$.

Fix a positive rational number $\ke$. 
A prestable  quasimap is called {\em $\ke$-stable} if the following two conditions hold:
\begin{enumerate}
\item The $\QQ$-line bundle
\begin{equation}\label{ample} \omega _{\uC} (\sum \ux _i) \ot (\phi _*([u]^*L _\theta ^{\ot \e})) ^{\ke/\e} \end{equation}
on the coarse curve $\uC$ is ample.

\item For every $x$ in $C$,  \[ \ke l(x)\le 1, \] where $l(x)$ is the length at $x$ defined 
in \cite[\S 7.1]{CKM}.
\end{enumerate}

A prestable quasimap is called {\em $(0+)$-stable} (or simply {\em stable}) if it is $\ke$-stable for every small enough positive rational number $\ke$.

\end{Def}


A few explanations are in order.

\begin{itemize}

\item Throughout this paper, a twisted curve is required to be balanced (\cite[\S 4]{AGV}).

\item The degree of $[u]^*L$ for $L\in \Pic \X$ is defined using  a finite covering of the normalization of $C$
as in \cite[\S 7.2]{AGV}.

Assume that the quasimap $[u]$ is prestable.
By the representability of $[u]$,  
the push-forward $\phi _*([u]^*L _\theta ^{\ot \e})$ is a line bundle on the coarse moduli space $\uC$  and the adjunction homomorphism 
$\phi^*\phi _*([u]^*L _\theta ^{\ot \e})\ra[u]^*L^{\ot \e}$ is an isomorphism, see Lemma 2.1.2 of \cite{AGV} and its proof. 
In particular, $\deg [u]^*L\in \frac{1}{\e}\ZZ$. In other words,
for prestable twisted quasimaps the class $\beta$ is an element of $\Hom _\ZZ(\Pic \X, \frac{1}{\e}\ZZ)$ and so it has uniformly bounded denominators.

However, note that the definition of the class $\beta$ as an element of $\Hom _\ZZ(\Pic \X,\QQ)$ makes sense for an arbitrary morphism of stacks $[u]:C\ra\X$ and we will use it later in this generality.

\item By its very definition (\cite[\S 7.1]{CKM}), the length at $x$ is nonzero if and only if $x$ is a base-point of the quasimap. By the prestable condition, these are away from the stacky points of the domain curve, hence it is appropriate to use the same notion of length to define stability in the orbifold theory as well. 

\item It is immediate from the definition that a prestable quasimap is $\ke$-stable if and only if 
\begin{enumerate}
\item for every irreducible component $C_i$ of $C$,
\begin{equation}\label{Bdd_Domain} 2g(C_i) - 2 + \#\text{ special points on }C_i + \ke \deg ([u]^* L_\theta |_{C_i})  > 0 \end{equation}
and 
\item $\ke l(x)\le 1$ for every point $x$ in $C$.
\end{enumerate}

\item For $(g, k)\ne (0, 0)$, a prestable quasimap is a stable twisted map into $X$
if and only it is an $\ke$-stable quasimap for some $\ke >1$. For $(g, k)=(0,0)$, the same is true with $\ke >2$.
For simplicity, a large enough $\ke$ will be denoted by $\infty$.

\item  The base locus of a prestable quasimap $[u]$ will be considered as the reduced scheme. 

\item The notions of (prestable, $\ke$-stable) quasimaps over an algebraically closed field of characteristic zero can be identically defined. 
\end{itemize}


\bigskip

From now on we fix $\ke \in [0+, \infty ]$. For short, we let $\bx = x_1, ..., x_k$.

\begin{Def}
A group homomorphism  $\beta \in \Hom _{\ZZ}(\Pic \X , \QQ )$ is called {\em $L_\theta$-effective} if it is realized as a finite sum of  classes 
of some quasimaps to $X$. Such elements form a semigroup with identity $0$, denoted by $\Eff (W, G, \theta )$.
\end{Def}

\begin{Lemma}
If $((C, \bx), [u])$ is a quasimap of class  $\beta$, then $\beta(L_\theta)\geq 0$. Moreover,  $\beta(L_\theta)=0$
if and only if $\beta =0$, if and only if the quasimap is constant (i.e., $u$ is a map into $X$, factored through an inclusion $B\Gamma
\subset X$ of the classifying groupoid $B\Gamma$ of a finite group $\Gamma$).
\end{Lemma}
\begin{proof}
Consider a  finite covering $\hat{C}$ of the normalization of $C$ such that $\hat{C}$ is a (possibly disconnected) nonsingular projective curve.
Then the induced map $[\hat{u}]:\hat{C} \ra \X$ is a union of quasimaps. Now the Lemma follows by applying \cite[Lemma 3.2.1]{CKM} to $[\hat{u}]$.
\end{proof}

\begin{Def}
Let $(C, \bx)$ be a family of twisted $k$-pointed genus $g$ curve over a scheme $S$, see \cite[\S 4]{AGV}, \cite[\S 1.1]{L-Ol}.
In particular, the markings $x_i\subset C$ are \'etale gerbes over $S$ banded by finite cyclic groups.
A pair $((C, \bx), [u])$ is called a (resp. {\em prestable, $\ke$-stable}) quasimap to $X$ over $S$ if
$[u]$ is a morphism from $C$ to $\X$ such that the restrictions to geometric fibers are 
(resp. prestable, $\ke$-stable) quasimaps in the sense of Definition \ref{Def_qmap}.
\end{Def}


 \begin{Lemma}\label{ext}
 Let $D$ be a codimension 1 nonsingular subvariety of a nonsingular variety $Y$, and let
$f$ be a morphism from $Y\setminus D$ to a separated DM stack $\cX$. 

Suppose that the coarse moduli level morphism $\underline{f}$ induced by $f$ is extendible to a morphism $Y\ra \underline{\cX}$. 
Then after possibly shrinking $Y$ at a given closed point $p$ of $D$ (i.e., taking an \'etale open subset containing  $p$) 
we have a {\em representable} morphism $\tilde{f}:
Y_{D, r} \ra \cX$ extending $f$, where $Y_{D, r}$ is the root stack
of the effective divisor $D$ with $r$-twisting for some $r$ (for the definition of root stacks see 
\cite{Cad, AGV}). The pair $(Y_{D, r}, \tilde{f})$ is unique up to a unique isomorphism, i.e.,
another one $(Y_{D,r'}, (\tilde{f})')$ is isomorphic to the $(Y_{D, r}, \tilde{f})$ up to a unique isomorphism. 
\end{Lemma}
\begin{proof} 
Since this is a local problem in the \'etale topology,
we may assume that $\cX = [Z/\Gamma]$ for some  affine scheme  $Z$ with a 
finite group $\Gamma$ action (see \cite[Lemma 2.2.3]{AV}, \cite[Proposition 4.2]{KM}). Take an analytic neighborhood $U$ of 
$p$ such that the fundamental group of $U\setminus D$ is isomorphic to $\ZZ$. 
Let $Z_{U\setminus D}$ be the principal $\Gamma$-bundle on $U\setminus D$ obtained by pulling back $Z\ra\X$ via $f$.
We have the monodromy action of $\ZZ$ on $Z_{U\setminus D}$, which factors through an action 
of the group $\mu _r$ of $r$-th roots of unity for some positive integer $r$ and induces a monomorphism $\mu _r \ra \Gamma$.

Take an $r$ to $1$ covering $U'\ra U$ branched along $D$, with $U'$ nonsingular. Let $D' \subset U'$ be the reduced divisor corresponding
to $D$. On $U'\setminus D'$, the pull-back bundle $Z_{U'\setminus D'}$ is a trivial $\Gamma$-bundle. Choosing a
section of it, we get by composition a $\mu_r$-equivariant morphism from $U'\setminus D'$ to $Z$.

By assumption, the induced map $$U'\setminus D'\lra\underline{[Z/\Gamma]}$$ is extendible to $$U'\ra U \ra \underline{[Z/\Gamma]}.$$
Since the coarse moduli space  $[Z/\Gamma]\ra \underline{[Z/\Gamma]}$ is proper by \cite{KM}, after shrinking $U'$ if necessary, 
there is a branched covering $U'' \ra U'$  along $D'$ and an extension $U''\ra [Z/\Gamma]$, where $U''$ is nonsingular.
Using a trivialization of the pullback bundle $Z_{U''}$, we obtain a morphism from $U''$ to $Z$. 
This implies that there is a morphism $U'\ra Z$ through which the morphism $U''\ra Z$ is factored.
The morphism $U'\ra Z$ is $\mu_r$-equivariant since it is so generically.
It therefore descends to a representable morphism $U_{D, r}:=[U'/\mu _r]   \ra [Z/\Gamma]$, as claimed in the Lemma. The uniqueness follows from
the separatedness of $[Z/\Gamma]$ and the uniqueness of $r$.
\end{proof}


 We note an immediate consequence of Lemma \ref{ext}.  If
$((C, {\bf x}), [u])$ is a prestable quasimap to $X$, with base locus $B$ (viewed as a reduced subscheme of $C$), then there is a canonical 
twisted curve $(C, {\bf x}\cup B)$ and a canonical representable morphism $[u]_{\reg} : (C, {\bf x}\cup B) \ra X $ 
which extends $[u]_{|_{C\setminus B}}$. Indeed, the required extension $\underline{[u]} :\uC\lra\uX$ exists by the prestable assumption and the fact that $\uX\lra\uX_0$ is projective.
As in Lemma 7.1.2 of \cite{CKM}, if $\beta _{\reg}$ denotes  the class of  $[u]_{\reg}$,
then  
\begin{equation}\label{beta reg} 
(\beta - \beta _{\reg})  (L_\theta ) = \sum _{x\in C} l(x) .\end{equation}

We will need also the following family version of the above consequence.

\begin{Cor}\label{ureg}
Let $\Delta$ be a nonsingular curve.
Consider a $\Delta$-family of  prestable  quasimaps $((C, x_1,..., x_k), [u])$ to $X$ with base locus $B$.
Then, after possibly shrinking $\Delta$ 
 and making an \'etale base change there is 
a unique $\Delta$-family of {\em twisted stable maps $(C', x_1',..., x_k', b_1', ..., b_l', [u]_{\mathrm{reg}})$ into} $X$,
together with an isomorphism $$\varphi: ((C' \setminus \bigcup \{b' _j\} _j ), x_1',..., x_k') \ra ((C \setminus B) , x_1,..., x_k)$$ over $\Delta$ making
$[u]\circ \varphi = [u]_{\mathrm{reg}}$ on $C'\setminus \bigcup \{b' _j \}_j $.  
\end{Cor}

\begin{proof} After shrinking $\Delta$ and \'etale base change, we may assume that the base locus $B$ forms (possibly empty) sections $b_j$, $j=1,..., l$, of $C\ra \Delta$, 
disjoint from the markings ${\bf x}$ and the nodes in the fibers.
Since $[u]$ is a morphism $C\ra \X$, it induces a morphism $C\ra \uX _0$, which is compatible with 
$C\setminus B \ra \uX $. Since $\uX\ra \uX _0$ is projective
and $B$ is a smooth divisor of $C$, we may extend $C\setminus B \ra \uX $ at the generic point of every component of $B$. So, after shrinking $\Delta $ again, 
we may assume that there is  an extension $C\ra \uX$ . Applying Lemma \ref{ext} concludes the proof.
\end{proof}

Now we come to the main result of this section. A related statement in the case of one specific GIT target is the main result of \cite{Deop}.

\begin{Thm}\label{Found_Thm}
The category fibered in groupoids  $\Qke$ of  genus $g$, $k$-pointed $\ke$-stable quasimaps to $X$ of
 class $\beta$ is a DM stack, proper over $\uX _0$. Further if $W$ is a locally complete intersection (LCI) variety, then $\Qke$ is equipped with a
canonical perfect obstruction theory.
\end{Thm}

Precisely speaking, $\Qke$ is a priori a $2$-category. An arrow from $(C, \bx, [u])$ to $(C' , \bx ', [u'])$ over a morphism $S\ra S'$ between schemes
is a pair $(\varphi, \alpha)$ of a cartesian product $\varphi: C\ra C'$ over $S\ra S'$ preserving the order of markings and
a 2-morphism $\alpha: [u]\Rightarrow [u']\circ \varphi $. 
A 2-arrow from $(\varphi, \alpha)$ to $(\varphi ', \alpha ')$ is a 2-morphism $\sigma : \varphi \Rightarrow \varphi '$ compatible with $\alpha$ and $\alpha '$.
Since $C, C'$ are DM stacks containing dense open algebraic spaces, the morphisms $\sigma$ are unique if they exist,
by \cite[Lemma 4.2.3]{AV}.  Therefore the 2-category $\Qke$ is equivalent to a category.

In fact, $\Qke$ depends on the pair  $(\X, X)$, see \cite[\S 4.6]{CKM}. Therefore, a more precise notation would be
$Q_{g, k}^\ke ((\X, X), \beta)$, but we'll only use the latter when needing to emphasize this feature.

%

When $\ke > 1$ (or $\ke >2$ if $(g,k)=(0,0)$), 
$\Qke$ is nothing but the stack $\cK _{g, k}(X, \beta )$ of $k$-pointed twisted stable maps into $X$ of genus $g$ 
and class $\beta$ introduced by Abramovich and Vistoli in \cite{AV}. We will also use the notation $Q_{g, k}^{\infty} (X, \beta):=\cK _{g, k}(X, \beta )$ for these moduli stacks.

\subsection{Proof of Theorem \ref{Found_Thm}}\label{Found_Proof}

In this section we prove Theorem \ref{Found_Thm}.

\subsubsection{Algebraicity}
Let $\fQmap_{g,k} (X, \beta )$ denote the category fibered in groupoids parameterizing $k$-pointed genus $g$ quasimaps to $X$ of class $\beta$ (no (pre)stability condition imposed).
We will show that $\fQmap_{g,k} (X, \beta )$ is an Artin stack locally of finite presentation over $\uX _0$.

Let $\fMtw$ be the  category fibered in groupoids of $k$-pointed genus $g$ twisted curves. It is proven in \cite[Theorems 1.9, 1.10]{L-Ol} that it
is a smooth Artin stack, locally of finite type. 
Let $\frC \ra \fMtw$ be the universal curve. 
We can view $\fQmap _{g,k} (X, \beta )$ as an open substack of the stack $\Hom _{\fMtw} (\frC  , \X \times  \fMtw)$
 whose fiber over a scheme $S$ is the groupoid of 
 $1$-morphisms from $S$-families of $k$-pointed genus $g$ twisted curves to $\X$ (for the definition of Hom-stacks see
 \cite[C.1]{Tw-AOV}, \cite[\S 2.3]{Lieb}).
  Hence the desired statement follows from the following

\begin{Lemma}
The Hom-stack $\Hom _{\fMtw}(\frC   , \X \times \fMtw)$  is an Artin stack
locally of finite presentation over $\CC$.
\end{Lemma}

\begin{proof}
Let $S$ be a scheme locally of finite type over $\CC$, with a smooth surjective
morphism $S\ra \fMtw$. Then the fiber product $S\times _{\fMtw}
\Hom _{\fMtw} (\frC   , \X \times \fMtw)$ is equivalent to $\Hom _S(\frC \times _{\fMtw}S, \X \times S)$,
which is an Artin stack  locally of finite presentation over $S$ by Proposition 2.11
in \cite{Lieb} since $S$ is an excellent scheme over $\CC$ 
and $\X\times S$ is the quotient stack $[W\times S / G]$. 
Thus, by Lemma C.5 in \cite{Tw-AOV} we conclude the proof. \end{proof}

\subsubsection{Automorphisms}
So far we have shown that $\Qke$ is an Artin stack. The proof of
 \cite[Proposition 7.1.5]{CKM} shows that
 a $\ke$-stable quasimap over a geometric point has no infinitesimal automorphisms. Indeed, the argument given there only involves the base locus of the stable quasimap, and is therefore insensitive to the stack structure of the domain curve. It follows that
the diagonal of $\Qke$ is formally unramified, and hence  $\Qke$
is Deligne-Mumford.  


\subsubsection{Boundedness}
For $d\in \frac{1}{\e}\ZZ_{\ge 0}$, 
we show that 
\[ \Qked := \coprod _{\beta: \beta (L_\theta) =d} \Qke , \] which is locally of finite presentation over $\CC$, is quasi-compact over $\CC$
and hence of finite type over $\CC$.

For a fixed twisted curve $D$ and $d\in \frac{1}{\e}\ZZ_{\ge 0}$, 
$\fQmap _{d} (D, X)$ will denote the substack of $\Hom (D, \X)$   
parameterizing quasimaps $[u]$ from $D$ to $X$ for which $\deg [u]^*(L_\theta )=d$.

(i) {\em Boundedness of $\fQmap _{d}(C, X)$}: Let $(C, {\bf x})$ be a twisted curve and choose a projective 1-dimensional 
variety $\tilde{C}$ with a degree $l_C$ finite flat morphism $\tilde{C} \ra C$ (see \cite{KV}).
By pullback, there is a natural morphism  $\Hom (C, \X)\ra \Hom (\tilde{C}, \X)$. 
Since   $\tilde{C}\ra C$ is fppf, by the effective descent for principal $G$-bundles and morphisms we see that
$$\Hom (C, \X)\cong  \Hom (\tilde{C}, \X)\times _{ \Hom (\tilde{C}\times _C \tilde{C}, \X)      }  \Hom (\tilde{C}, \X). $$
On the other hand, by Theorem 3.2.4 of \cite{CKM} (applied to the normalization of $\tilde{C}$) the moduli stack $\fQmap_{l_{C}d} (\tilde{C}, X)$ 
is quasi-compact. It follows that the stack $\fQmap_{d}(C, X)$ is quasi-compact.

(ii) {\em Boundedness of the domain curves}: 
The boundedness of topological types of coarse moduli spaces  of possible domain curves follows from \eqref{Bdd_Domain}
and hence we obtain the boundedness of the possible domain twisted curves.

\medskip

Now the quasi-compactness of $\Qked$ follows from (i) and (ii):
First, for each $(C, \bx) \in \fMtw (\CC )$, using (i), we take a finite type scheme $U_{(C, \bx )}$ 
smooth over $\Qked$ containing  all quasimaps  with the domain curve $(C, \bx)$.
Next, by (ii) there is a  finite collection  $\{ U_{(C, \bx)} \}_{(C, \bx)}$ surjectively covering $\Qked$.

\begin{Rmk}
The boundedness of $\Qked$ implies that there are only finite many $\beta \in \Eff (W, G, \theta )$ with $\beta (L_\theta ) =d$.
\end{Rmk}


\subsubsection{Properness}
To prove the properness of $\Qke \ra \uX _0$, 
it is enough to check the valuative criterion with discrete valuation rings since the stack is already shown to be of finite type over $\CC$. 
It is straightforward to see that the valuative criterion for separatedness follows from 
an argument identical to the one given in \cite[\S 4.1] {CK}.

Once we know the existence of $[u]_{\reg}$ from Corollary \ref{ureg}, the valuative criterion of properness
can be checked  by the proof of Theorem 7.1.6 in  \cite{CKM}. The argument
requires the properness of the moduli stack $\cK _{g, n}(X, \beta)$ of twisted stable maps, which is due to Abramovich and Vistoli \cite{AV}.
There is a slight modification as follows.
In the proof of \cite[Theorem 7.1.6]{CKM}, we need to contract unstable rational tails attached to the central fiber 
of a completed twisted curve $\widehat{C}$ over a nonsingular curve $(\Delta, 0)$. Those rational tails might have stacky nodal 
points. When we contract such unstable rational tails, we remove the stack structure 
of such nodes to make the rational tails ($-1$)-curves and then we contract those ($-1$)-curves.

\subsubsection{Obstruction Theory}
This part is identical to the corresponding one in \cite{CKM}.
If we write $$\fB un_\G^{\tw} := \Hom _{\fMtw} (\frC  , B\G  \times {\fMtw})$$ which is smooth over $\fMtw$ since the obstruction
vanishes (see, e.g., \cite[Proposition 2.1.1]{CKM} for the vanishing of the obstruction), there is a natural forgetful morphism 
 $\sigma: \Qke \ra \fB un _\G ^{\tw} $.  
 We describe a canonical $\sigma$-relative obstruction theory.
 
 Let $\pi : \mathcal{C}\ra \Qke$ be the universal curve. The universal morphism $[u]: \cC \ra \X = [W/G]$ induces
the universal principal $G$-bundle $\cP \ra \cC$ by the pullback of the principal $G$-bundle $W\ra [W/G]$ 
(here $\cP$ is an algebraic space over $\CC$ since $[u]$ is
representable). In turn, this determines the  universal fiber bundle $$\rho : \cP \times _\G W := [\cP \times W /G] \ra \cC$$
and the {\em universal section} $u: W \ra \cP \times _G W$ of $\rho$. This is summarized in the diagram:
\[ \xymatrix{ \cP \times _G W \ar[r]_{\ \ \ \rho} & \cC \ar[r]_{\pi \ \ \ \ \ } \ar[d]_{[u]} \ar@/_0.7pc/[l]_{\ \ \ \ u} & \Qke \ar[r]_{\ \ \ \sigma} & \fB un_\G^{\tw}  \\
                                                                        & [W/G]  .             &                                    &  } \]

 Let ${\mathbb T}_\rho$ be  the relative tangent complex of $\rho$. The  $\sigma$-relative obstruction theory
is given by the complex $(R^{\bullet}\pi _* u^* {\mathbb T}_\rho  )^\vee $.
As shown in Theorem 4.5.2 of \cite{CKM}, this complex is two-term perfect
if $W$ is LCI and $W^{s}$ is smooth.

\medskip
{\em From now on we will assume that $W$ is LCI and $W^s=W^{ss}$ is smooth}.
\medskip

Note that there are quasi-isomorphisms
$$R^\bullet\pi _* \left( \cP \times _G \mathfrak{g} \cong u^*(\cP \times _G (\mathfrak{g}\ot \cO _W))  \right) \simeq (\sigma ^* \mathbb{L}_{\fB un _G ^{\tw}/ \fMtw } [1])^\vee $$
and $\mathbb{T}_\rho \simeq \cP \times _G \mathbb{T}_W$.

Note also that on $W$ there is a natural distinguished triangle $$\mathfrak{g}\ot \cO _W \ra \mathbb{T} _W \ra \mathbb{T} _{[W/G]}|_{W}.$$
Therefore,
we obtain the perfect obstruction theory $$(R^\bullet \pi _* [u]^* \mathbb{T}_{[W/G]})^\vee$$
for $\Qke$ relative to the pure dimensional, smooth stack $\fMtw$ as in \cite[\S 5.3]{CK}. 
Further, the two relative perfect obstruction theories determine the same absolute perfect obstruction theory on $\Qke$ and the virtual classes associated to all three obstruction theories coincide.


\subsection{Basic Properties and Variants}

\subsubsection{Expected dimension} 
Since the \'etale gerbe markings are away from base locus for $\ke$-stable quasimaps, 
there are natural evaluation morphisms
\[ ev_i : Q^\ke _{g, k}(X, \beta ) \ra \bIX ,\  ((C, x_1, ..., x_k), [u]) \mapsto [u]_{|_{x_i}} \ \ \ i=1, ..., k,\]
to  the rigidified cyclotomic inertia stack $\bIX$ of the DM-stack $X$, as explained in \cite[\S 4.4]{AGV}. 
The stack $\bIX$, which  parameterizes representable maps from gerbes banded by finite cyclic groups to $X$,
is a smooth (resp. proper) stack over $\CC$ if $X$ is a smooth (resp. proper) stack over $\CC$ (see \cite[Corollary 3.4.2]{AGV}).

Let $\coprod _{c\in R} X_c$ denote  the connected component decomposition of $\bIX$ for some index set $R$
and, for $c_i \in R, i=1, ..., k$, let
\[ Q^{\ke}_{g, k}(X, \beta; c_{1},..., c_k ) := (\prod _i ev_{i} )^{-1} ( \prod _{i} X_{c_i}).  \]
Its virtual dimension is 
\begin{equation}\label{dim}  k+(1-g)(\dim X - 3) + \beta(\det \mathbb{T}_{[W/G]}) - \sum _{i=1}^k \mathrm{age} X_{c_i} \end{equation}
by Riemann-Roch for twisted curves (\cite[Theorem 7.2.1]{AGV}).
The age is defined as follows. Let $(\bar{x}, g)$ be a geometric point of $X_{c}$, with $r$ the order of $g\in \mathrm{Aut}(\bar{x})$.
The age of $X_c$ is
\[ \sum _{j=0}^{r-1} \frac{j}{r}\dim E_j \]  where $E_j$ is the eigenspace of $T_{\bar{x}}X$ of the induced $\mu _r$-action with
eigenvalue $e^{2\pi \sqrt{-1} j/ r }$.

\subsubsection{Trivializations of gerbe markings}

As in \cite[\S 6.1.3]{AGV}, one may construct 
the moduli stacks whose objects over a scheme $S$ are $S$-families of  $\ke$-stable quasimaps {\em with sections of the gerbe markings}.
We will not use these stacks.

%
%

\subsubsection{Graph spaces}

For $g,k\ge  0$ and $\ke\geq 0+$, define the graph space $QG^{\ke}_{g, k, \beta}(X)$ to be the moduli
stack of {\em $\ke$-stable graph quasimaps}. By a (resp. prestable) graph  quasimap we mean the data  $$((C, x_1, ..., x_k), P, [u]:=([u]_1, [u]_2) ) $$ 
with $((C, x_1, ..., x_k),  [u]_1)$ a $k$-pointed, genus $g$ (resp. prestable) quasimap to $X$ 
and a map $[u]_2: C \ra \PP ^1$ for which $\underline{[u]_2} : \uC \ra \PP ^1$ is a degree $1$ map,
that is, there is a unique component of $\uC$ isomorphic to $\PP ^1$ under $\underline{[u]_2}$.
For $\ke\in \QQ_{>0}$, the $\ke$-stability for a prestable graph quasimap is defined by imposing the requirements that
\begin{equation}\label{graph-stability}\omega _{\uC} ( \sum \underline{x}_i)\ot  (\phi_* ([u]_1^* L_\theta )^\e) ^{\ke/\e} \ot  \underline{[u]_2}^* \mathcal{O}_{\PP ^1} (3) \;\; \text{is ample}\end{equation}
and that 
\begin{equation}\label{graph-stability2}\ke l(x) \le 1,\;\;\text{for every $\CC$-point}\; x\in C.\end{equation} 
Again, $l(x)$ in \eqref{graph-stability2} is the length of the quasimap at $x$ defined in \cite[\S 7.1]{CKM}. 
When the requirement \eqref{graph-stability} is true for every small enough $\ke \in \QQ_{>0}$, we say that the graph quasimap is {\em $(0+)$-stable} 
(the length inequality imposes no condition and is discarded for $\ke=0+$).

For $\ke \in [0+, \infty )$, by the same argument as in \S\ref{Found_Proof}, $\QGke$ is a finite type, proper DM-stack over $\uX _0$. 
The universal family of $\ke$-stable graph quasimaps consists of the universal principal bundle $\cP$ on the universal curve $\pi :\cC\lra QG^{\ke}_{g, k, \beta}(X)$
and the universal map 
$$u=(u_1,[u_2]) :\cC\lra (\cP \times _\G W) \times \PP ^1,$$
with $u_1$ the universal section of $\rho : \cP \times _\G W\lra \cC$.
If we let $$\tilde{\rho} :  (\cP \times _\G W) \times \PP ^1 \ra \cC$$ be the composition of $\rho$ with the projection to the first factor,
then the perfect obstruction theory relative to
$\fB un _{\G} ^{\tw}$ remains of the
same form $ (R^{\bullet}\pi _* u^* {\mathbb T}_{\tilde{\rho}}  )^\vee$.

\subsubsection{Localization}

Assume $W$ has an action by an algebraic torus $T$, commuting with 
the $G$-action. This induces $T$-actions on $\cX$, $X$, $\uX_0$, $\bIX$, and on the moduli spaces $\Qke$. 
In this situation, there is a $T$-equivariant embedding of $\Qke$ into a
smooth Deligne-Mumford stack. The argument from  \S 6.3 of \cite{CKM} works. Therefore we may
apply the virtual localization theorem of \cite{GP} to $\Qke$ and $QG^{\ke}_{g, k, \beta}(X)$.

\subsubsection{Orbifold quasimaps with weighted markings}\label{weighted}

The results of \S \ref{Def_Qmaps}, \S \ref{Found_Proof} also provide the extension to orbifold targets of the theory of quasimaps
with weighted markings developed in \cite{BigI}.

 Let $\theta _0$ be
the minimal {\em integral} character on the ray $\theta \cdot \QQ_{>0} \subset \chi (G)_{\QQ}$
in the character group of $G$  with $\QQ$-coefficients. Then the
$\ke$-stability condition on quasimaps can be interpreted as stability with respect to
the {\em rational}  character $\ke\theta _0$, see 
Definition 2.6 and Proposition 2.7 (i) in \cite{BigI}.

Now for $\ke\in \QQ_{>0}\cup \{0+\}$ and $\delta _1, \dots ,\delta _m \in
(\QQ_{>0}\cap (0, 1])\cup \{0+\}$, consider the rational character
$$\uptheta  := \ke \theta _0+ \sum \delta _j \mathrm{id}_{\CC ^*}$$ of the
group $G\ti (\CC^*)^m$. Then the moduli space of genus $g$, $\ke$-stable quasimaps of class $\beta$
to $X$, with $k$ usual markings and $m$ markings weighted by $\delta _1,\dots , \delta _m$,  is 
identified with the moduli stack 
$$Q_{g, k}^{\uptheta} ([W^{ss} /G]\times [\CC ^{ss}/\CC^*]^m , (\beta, 1,\dots,1))$$
constructed in \S \ref{Def_Qmaps}.
Note that the domain curves carry nontrivial stack structure only at nodes and at the usual markings.

An important special case considered in \cite{BigI} is when
$\delta _{1} = \dots= \delta _m =\delta$ for some $\delta\in(\QQ_{>0}\cap (0, 1])\cup \{0+\}$. This gives a theory with {\em two} stability parameters $(\ke,\delta)$, for which the corresponding
moduli stacks are denoted $Q^{\ke,\delta}_{g,k|m}(X,\beta)$.

Similarly, there are graph spaces with weighted markings. 

By Theorem \ref{Found_Thm}, the moduli spaces with weighted markings are DM stacks, proper
over the affine quotient, and carrying the canonical perfect obstruction theory for LCI $W$.



\section{Quasimap Theory}


\subsection{Quasimap Invariants}

We extend the quasimap theory for the orbifold case $(\X = [W/G], X=[W^{ss}/G])$, closely following \cite{AGV, CKM, CKg0}.
Fix an algebraic torus $\T$ action on $W$, commuting with the given $\G$ action. We allow the case
when $\T$ is the trivial group. Assume that the $T$ fixed locus $\uX _0^T$ of the affine quotient $\uX_0$ is a finite set of points.
Let $K:=\QQ(\{\lambda_i\})$ be the rational localized $T$-equivariant cohomology of $\Spec\CC$, with $\{\lambda_1,\dots,\lambda_{\rank(T)}\}$ corresponding to a basis for the characters of $T$.
The Novikov ring is defined to
be \[ \Lambda _K := K [[\Eff (W, G, \theta ) ]] .\] 
We write $q^\beta$ for the element corresponding to $\beta$ in $\Lambda _K$ 
so that $\Lambda _K$ is the $q$-adic completion. We denote  by  $q\Lambda _K$ the maximal ideal generated by
$q^\beta$, $\beta \ne 0$.

Let $\{ \gamma _i\}$ be a basis of the $T$-equivariant {\it Chen-Ruan cohomology} of $X$,
$$H^*_{\mathrm{CR}, T} (X, \QQ):=\HA ^*_T(\bIX , \QQ ) $$ and
let $\{\gamma ^i\}$ be the dual basis \lq\lq with respect to the Poincar\'e pairing in the {\it non-rigidified }cyclotomic inertia stack $I_\mu X$ of $X$" in the sense that
$$\lan \gamma _i , \gamma ^j \ran _{\orb} := \int _{\sum  _{r\in \NN_{\ge 1} } r^{-1} [\bar{I}_{\mu _r}X]} \gamma _i \cdot \iota ^* \gamma ^j  = \delta _i ^j ,$$
with $\iota$ the involution of $\bIX$ obtained from the inversion automorphisms.
Note that $$\sum _{r=1}^{\infty}  r[ \Delta _{\bar{I}_{\mu _r}X}] = \sum _i \gamma _i \ot \gamma ^i \text{ in } H^*(\bIX \times \bIX , \QQ), $$
where the diagonal class $[\Delta _{\bar{I}_{\mu _r}X}] $ is obtained via push-forward of the fundamental class by $(\mathrm{id}, \iota): \bar{I}_{\mu _r}X \ra \bar{I}_{\mu _r}X \times \bar{I}_{\mu _r}X$.

Define $\psi_i$ to be the first Chern class of
the universal cotangent line whose fiber at $((C, x_1, ..., x_k), [u])$  is the cotangent space of
the coarse moduli $\underline{C}$ of $C$ at $i$-th marking $\ux _i$.  
For non-negative integers $a_i$ and  classes $\alpha _i \in \HA ^*_T (\bIX  , \QQ)$,  $t  = \sum _j t_j \gamma _j$ with formal variables $t_j$,
we write
\begin{align*}   \lan \alpha _1\psi ^{a_1}, ..., \alpha _k \psi ^{a_k} \ran^\ke _{g, k, \beta}   & :=
\int _{[Q^\ke_{g, k}(X, \beta )]^{\vir}}  \prod _i  ev _i ^* (\alpha _i) \psi ^{a_i}_i   ;\\ 
 \lla  \alpha _1\psi ^{a_1}, ..., \alpha _k \psi ^{a_k}   \rra_{k, \beta}^{\ke} 
       & := \sum _{m\ge 0}  \frac{1}{m!}  \lan \alpha _1\psi ^{a_1}, ..., \alpha _k \psi ^{a_k}, t, ..., t \ran^\ke _{g, k+m, \beta} ; \\ 
       \lla \alpha _1\psi ^{a_1}, ..., \alpha _k \psi ^{a_k}  \rra_{k}^{\ke} 
          & := \sum _{\beta, m}  \frac{q^\beta}{m!}   \lan \alpha _1\psi ^{a_1}, ..., \alpha _k \psi ^{a_k}, t, ..., t \ran^\ke _{g, k+m, \beta}\\ 
        & \in \Lambda _{K} [[\{t_j \}_j]]. \end{align*}

We may also define quasimap Chen-Ruan classes.
Write 
 \begin{equation} \label{tilde ev}
 (\ev _j)_* = \iota _* (\br _{j}(ev_{j})_*), 
 \end{equation}
where $\br _{j}$ is the order function of the band of the gerbe structure at the marking $j$.
Define a class in $H_*^T (\bIX )\cong H^*_T (\bIX)$ by
\begin{align*}  \lan \ka _1, ..., \ka _k, - \ran ^{\ke}_{g, \beta} &:= (\ev _{k+1})_* \left( ( \prod ev _i ^* \ka _i) \cap [\Qke]^{\vir}\right)  \\
                                    &= \sum _i \gamma ^i \lan \ka _1, ..., \ka _k, \gamma _i \ran ^{\ke}_{g, k+1, \beta }. \end{align*}
Since the evaluation maps are proper, these are well-defined without $T$-localization, even when the coarse moduli space $\uX$ is not projective.

\subsection{$J^\ke$-function}

As in \cite{CK, CKg0}, we use $\CC^*$-residues on graph spaces to define $J^\ke$-functions. However, simply copying \cite[Definition 5.1.1]{CKg0} in the orbifold setting would give a function with
values in the untwisted sector $H_T^*(X)$.  
To ensure that the $J^\ke$-functions take values in the full Chen-Ruan cohomology of $X$, including the twisted sectors, we will use instead
graph spaces with an extra marking.

First of all, we fix a $\CC ^*$ action on $\PP ^1$ given  by
$$t[\zeta_0,\zeta_1]=[t\zeta_0,\zeta_1],$$
for $t\in \CC ^*$, $[\zeta _0, \zeta _1] \in \PP ^1$.
We consider $QG^{\ke}_{g,k, \beta}(X)$ with the induced $\CC^*$ action as well as the induced $T$ action. This gives rise to
the canonical $ T\times \CC^*$-equivariant perfect obstruction theory on $QG^{\ke}_{g,k,\beta}(X)$ via the action of $ T\times \CC^*$ on the tangent complex
$ \mathbb{T}_{ \X} \boxtimes T_{\PP ^1} $.

We use the same notation for the evaluation maps on graph spaces
\[ ev_i : QG^{\ke}_{g,k,\beta}(X) \ra \bIX \times \PP ^1, \ \ \ i=1, ..., k.\]
Let $z$ denote the $\CC^*$-equivariant Euler class of the line bundle associated to the fundamental representation space
of $\CC^*$.

We define $T\times\CC^*$-equivariant invariants on the graph spaces as follows: for $\sigma _j \in \HA_T^*(\bIX) \otimes \HA ^*_{\CC ^*}(\PP ^1)$, 
$t \in \HA _T^*(\bIX ) \cong \HA_T^*(\bIX ) \ot 1  \subset \HA_T^*(\bIX ) \otimes \HA ^*_{\CC ^*}(\PP ^1) $
\begin{align*} \lan \sigma _1, ..., \sigma _k \ran _{k, \beta}^{QG^\ke}  & : = \int _{[QG^\ke _{0, k,\beta}(X)]^{\vir}} \prod _i  ev_i^* (\sigma _i) ;\\
      \lla \sigma _1, ..., \sigma _k \rra_{k, \beta}^{QG^\ke} & := \sum _{m\ge 0}  \frac{1}{m!}  \lan \sigma _1, ..., \sigma _k, t, ..., t \ran _{k+m, \beta}^{QG^\ke} ;\\
       \lla \sigma _1, ..., \sigma _k \rra_{k}^{QG^\ke} & := \sum _{\beta, m}  \frac{q^\beta}{m!}  \lan \sigma _1, ..., \sigma _k, t, ..., t \ran _{k+m, \beta}^{QG^\ke}.
       \end{align*}
       The integration is understood as equivariant push-forward to $\Spec\CC$, hence $\lan \sigma _1, ..., \sigma _k \ran _{k, \beta}^{QG^\ke} \in K[z]$ and the double brackets
       are elements of $K[[z]][[\mathrm{Eff}(W,\G,\theta)]][[\{t_i\}_i]]$.

In what follows, for an integer $k\ge 0$, $[k]$ will mean the index set $\{1, ..., k\}$.
 Let $F := \QGs ^{\CC ^*} $ denote  the $\CC ^*$-fixed substack of $\QGs$ 
(the maximal closed substack fixed by the action of every element of $\CC ^*$).
If  $(C, x_\star, \bx , [u])$ is a $\CC ^*$-fixed  $\ke$-stable graph quasimap, then
on the coarse curve of the distinguished component of $C$ we must have that $\underline{[u]_{\reg}}$ is a constant map to $\uX$.
The base locus of $[u]$, if nonempty, must be over $0:=[0,1]$ or $\infty:=[1,0]$ in the target $\PP ^1$. The same is true for the markings $(x_\star, \bx)$.

As in \cite[\S 4.2]{CKg0}, there are \lq\lq components" $F^{A_1, \beta _1}_{A_2, \beta _2}$ of $F$ corresponding to distributions of the markings and of the class $\beta$ over $0$ and $\infty$:
for $A_1\coprod A_2=[k]\cup\{\star\}$ and $\beta_1+\beta_2=\beta$, the closed substack $F^{A_1, \beta _1}_{A_2, \beta _2}$ parameterizes $\CC^*$-fixed $\ke$-stable graph quasimaps for which  markings in $A_1$ and the class $\beta_1$ are supported over $0$, while $A_2,\beta_2$ are supported over $\infty$.
In particular, we have a distinguished $\CC ^*$-fixed part $ F^{k, \beta}_{\star, 0}$
of $\QGs$. 
This closed substack  $ F^{k, \beta}_{\star, 0}$ consists of
$\ke$-stable graph quasimaps $((C, x_{\star}, \bx), [u] )$ such that only the marking $x_{\star}$ is over $\infty$ while the markings in $[k]$ and the entire class $\beta$ are over $0\in \PP ^1$.

Let $\eta _0$, $\eta _\infty$ be the $\CC ^*$-equivariant classes of $\PP ^1$ defined by the
property $$\eta _0 |_0 = z,\;\; \eta _\infty |_\infty = -z,\;\; \eta _0 | _{\infty} =0=\eta _\infty |_0.$$

\begin{Def} 1. Define the $J^\ke $-function by
\begin{align*}  J^{\ke} (t, q, z)  & := \sum _{\beta \in \Eff (W, G, \theta ), k\ge 0} q^\beta  (\widetilde{\pr _{\bIX}  \circ ev_\star})_*  \\
& \left(  (ev _\star) ^*(\eta _\infty) \frac{ \prod _{i=1}^k  ev _i ^* ( t ) }{k!} \frac{[F^{k, \beta}_{\star , 0}]^{\vir}}{e^{\CC ^*\times T} (N^{\vir}_{F^{k, \beta}_{\star , 0}/
QG ^\ke_{0, [k]\cup\{\star\}, \beta }(X)} )}  \right)   , \nonumber \end{align*}
where $\pr _X$ denotes the projection to $X$ from $X\times \PP ^1$ and the notation $(\widetilde{\pr _{\bIX}  \circ ev_\star})_*$ is as defined in \eqref{tilde ev}. 
The localization residue is taken as a sum over the 
connected components of $F^{k, \beta}_{\star , 0}$.

2. Define the $S^\ke $-operator by
\begin{align*}  &  S^{\ke} _t (z) (\gamma ) 
              := \sum _i  \gamma _i \lla \frac{\gamma ^i}{z-\psi } , \gamma \rra ^{\ke}_{0, \{\bullet, \star\} }                  \\
                                               & := \gamma + \sum _{\beta\ne 0 \, \mathrm{ or } \, k\ge 1} q^\beta  (\ev_\bullet)_*\left( \frac{(ev _\star)^* (\gamma )}{z-\psi _\bullet} 
                                               \frac{\prod _{i=1}^k ev_{i}^*(t) }{k!}  \cap [Q^\ke _{0, [k]\cup \{\bullet, \star\} }(X, \beta )]^{\vir}\right) 
                                            \\ & =   \gamma +  O (1/z) .               
                                               \end{align*}
\end{Def}

\begin{Rmk}\label{mod q}
\begin{enumerate}

\item
It is straightforward to check that  the definitions of $J^\ke$ and $S^\ke$ agree with those given for the same objects  in \cite{CK, CKg0} when $X$ becomes a scheme.

\item As noted before, properness of the evaluation maps implies that $J^\ke$ and $S^\ke$ are well-defined without $T$-localization for all targets.  

\item When $\beta=0$ the moduli spaces coincide for different values of $\ke$. 
Hence $J^{\ke} (t, q, z) $, $S^\ke_t (z) $ modulo $q$ do not depend on $\ke$.

\item  Consider Givental's symplectic space $$\mathcal{H}_{X}:=\HA_T^*(\bIX)\otimes_{\QQ} K((z^{-1})) [[\Eff (W, G, \theta ) ]].$$ 
The series $J^\ke$ is an element of $\mathcal{H}_{X}[[\{t_j\}_j]]$ and the operator $S_t^\ke$ may be viewed as a family (parametrized by the formal variables $t_j$) of endomorphisms of $\mathcal{H}_{X}$.

\end{enumerate}

\end{Rmk}


\subsection{Factorization}

For  an Artin stack $\mathcal{Y}$  over $\CC$, a twisted curve $C$ (with some gerbe markings) over $\CC$
and $\beta \in \Hom _{\ZZ} (\Pic \mathcal{Y}, \QQ )$,
we denote by $\Hom _{\beta}^{\mathrm{rep}} (C, \mathcal{Y})$
 the substack of $\Hom _{\Spec\CC}(C, \mathcal{Y})$ consisting of
representable $1$-morphisms $C\ra \mathcal{Y}$ with class $\beta$.

Let $\mu _r$ be the group of $r$-th roots of unity. 
We will consider a special twisted curve $$C_{r, -r}:=[\PP ^1/\mu _r]$$ where $\mu _r$ acts $\PP^1$ 
by $s\cdot [\zeta _0, \zeta _1] = [s\zeta _0, \zeta _1]$ for $s\in \mu _r$, $[\zeta _0, \zeta _1]\in \PP ^1$.

\begin{Lemma}\label{simple}  
Let $Y$ be a DM stack over $\CC$.
\begin{enumerate} 
\item 
Consider a representable morphism $[u]:=([u]_1, [u]_2)$ from a genus $0$ irreducible twisted curve $(C, x_1, x_2)$ over a scheme $S$ to $Y\times \PP ^1$
sending $x_1$ to $0$ and $x_2$ to $\infty$ under $[u]_2$, with class $(0, [\PP ^1])$. Then  for some positive integer $r$,
$C$ is canonically isomorphic to $C_{r, -r}\times S$ with $[u]_2$ as the coarse moduli space. 
 
\item  The closed substack $(\pr _{\PP ^1}\circ ev _1 )^{-1} (0) \cap (\pr _{\PP ^1}\circ ev_{2}) ^{-1} (\infty )$ of $\cK _{0, 2}(Y\times \PP ^1, (0, [\PP^1]))$
is isomorphic to the Hom stack $$\Hom ^{\mathrm{rep}}_{0}(C_{r,-r}, Y)$$ parameterizing 
representable $1$-morphisms from $C_{r,-r}$ to $Y$ with class $0$.

  \item The stack $\Hom ^{\mathrm{rep}}_{0}(C_{r,-r}, Y)$ 
 is isomorphic to  the cyclotomic inertia stack $\IY$ of $Y$, compatible with the rigidification map $\IY\ra \bIY$ and the two evaluation maps (with the involution).
 Here $\pr _{\PP ^1}$ denotes the projection to $\PP ^1$ from $Y \times \PP ^1$.
 
\end{enumerate}
\end{Lemma}

\begin{proof}
For 1): By \cite[Theorem 4.2.1]{AGV}, it is enough to show the case when $S=\Spec \CC$.
We may assume that $Y = B\Gamma$ for some finite group $\Gamma$. Denote by $\bar{x}_i$ the geometric points
associated to the gerbes $x_i$.
The monodromies of the induced $\Gamma$-bundle on an \'etale chart of $C$ around $\bar{x}_i$
are  exactly the homomorphisms $\mathrm{Aut}(\bar{x}_i)\ra \mathrm{Aut}(\overline{Y}) = \Gamma$, where
$\overline{Y}$ is the $\CC$-point of $Y$. Since the homomorphisms are monomorphisms, we are done.

For 2): There is a natural functor $\Phi$ from $\Hom ^{\mathrm{rep}}_{0}(C_{r,-r}, Y)$ 
to $\mathcal{D}:=(\pr _{\PP ^1}\circ ev _1 )^{-1} (0) \cap (\pr _{\PP ^1}\circ ev_{2}) ^{-1} (\infty )$.
It is essentially surjective by (1). It is clearly faithful. Given 
an arrow $(\varphi, \alpha)$ in $\mathcal{D}$ between objects $(C_{r, -r}\times S, [u])$ and 
$(C_{r, -r}\times S, [u'])$ over a scheme $S$, i.e., 
an $1$-isomorphism $\varphi: C_{r,-r}\times S \ra C_{r,-r}\times S$ over $S$ with a 2-morphism $\alpha: [u]\Rightarrow [u']\circ \varphi$,
first note that there is a unique 2-morphism $\sigma : \varphi \Rightarrow \mathrm{id}_{C_{r,-r}\times S}$.  Now define $\alpha ' := [u'](\sigma ) \circ \alpha$ so that
$\sigma : (\varphi, \alpha)\Rightarrow (\mathrm{id}_{C_{r,-r}\times S}, \alpha ')$. This shows
that the functor $\Phi$ is full.

For 3):  Every object  $[\PP ^1/\mu _r] \times S \ra Y$ in $\Hom ^{\mathrm{rep}}_{0}(C_{r,-r}, Y)$ is factored through 
a unique representable morphism $B\mu _r\times S \ra Y$ 
given by one of sections of the inclusions $x_1, x_2 \subset [\PP ^1/\mu _r]\times S$.
This yields a functor  $\Hom ^{\mathrm{rep}}_{0}(C_{r,-r}, Y) \ra \IY$. In fact, this functor is an equivalence 
according to the definition \cite[Definition 3.2.1]{AGV} of $\IY$.
\end{proof}

Our next goal is to prove the following factorization of graph 
quasimap double brackets.

\begin{Prop}\label{factorization_Prop}
Let $\gamma, \sigma \in \HA ^*_T(\bIX )$. Then  
\begin{align} \lla \gamma \ot \eta _0 , \delta \ot \eta _\infty \rra _{\{\star_1, \star _2\}} ^{QG^\ke}
 = \sum \lla\gamma ,  \frac{\gamma ^i}{z-\psi} \rra _{0, \{\star _1, \bullet \} } ^\ke \lla \frac{\gamma _i}{-z-\psi}, \delta  \rra _{0, \{ \check{\bullet} , \star _2 \} } ^\ke  \label{fact} \end{align}
 \end{Prop}
 
We will use $\CC ^*$-localization, as in the proof Proposition 5.3.1 of \cite{CKg0}.
To this end, we first establish a factorization of the virtual classes of the $\CC^*$-fixed loci 
$F^{A_1, \beta _1}_{A_2, \beta _2}$.

We use $0$ and $\infty$ to label the  two markings $\star _1, \star _2$ by abusing notation.
Let $D^{\tw}(A_1, A_2)$ denote the locally finite type smooth Artin stack parameterizing genus $0$ nodal twisted curves with 
two {\em distinguished nodes}, splitting the curve into three twisted curves $(C_1, A_1 \cup\{  {\bullet} \}) $, 
$(C_{\br _{\bullet}, -\br_{\check{\bullet}}}, \{ \bullet, \check{\bullet}\})$, and  $(C_2, A_2 \cup \{ \check{\bullet}\})$, 
together with sections of the gerbes $\bullet, \check{\bullet}$
(see \cite[\S 5.1]{AGV}). 
Here $(C_{\br _{\bullet}, -\br_{\check{\bullet}}}, \{ \bullet, \check{\bullet}\})$ denotes $C_{r, -r}$ for $r$=the index of
the band of the gerbe $\bullet$.

Let $A_1\coprod A_2 = [k]\cup \{ 0, \infty \}$.
Assume $0\in A_1$, $\infty \in A_2$. 
Below, the unstable terms $Q^\ke_{0, \{0, \bullet\}, \beta _1=0}(X)$, $Q^\ke_{0, \{\infty, \check{\bullet}\}, \beta _2=0}(X)$ are 
understood as $\bIX$.

Consider the commuting diagram:

 \[ \xymatrix{ Q(A_1, A_2, \beta ) \ar[r]^{\ \ \ \Psi _1}_{\ \  \ \cong}  \ar[d]       &    F_{A_1, A_2, \beta}   \ar[d] \ar[r] &  QG_{\beta}^{\ke} \ar[d] \\
     D^{\tw}(A_1, A_2)    \ar[r]^{\ \ \ \Psi _2}_{\ \  \ \cong}                &   \fM _{A_1, A_2} ^{\tw, \CC ^*}   \ar[r]    & \fM _{0, k} ^{\tw} (\PP ^1, [\PP ^1]) , }\] 
where
 \begin{align*}  Q(A_1, A_2, \beta ) & := \coprod _{\beta _1 + \beta _2 = \beta }  \Qa \times _{\bIX} \IX \times _{\bIX}  \Qb ; \\
F_{A_1, A_2, \beta} & := \coprod _{\beta _1 + \beta _2 = \beta } F ^{A_1, \beta _1}_{A_2, \beta _2} ; \\
   QG_{\beta} ^{\ke} & := QG ^{\ke} _{0, k, \beta } (X ).
 \end{align*}
Further, $$ \fM _{0, k} ^{\tw} (\PP ^1, [\PP ^1]) \subset  \Hom _{\fM _{0, k}^{\tw}}(\mathfrak{C}, \PP ^1\times \fM _{0, k}^{\tw} )$$ 
is the (locally finite type and smooth) Artin stack of $k$-pointed, genus $0$, 
twisted (not necessarily stable) maps to $\PP ^1$ with class $[\PP ^1]$,
while $ \fM _{A_1, A_2} ^{\tw, \CC ^*} $ is the part of the $\CC^*$-fixed closed substack of $\fM _{0, k} ^{\tw} (\PP ^1, [\PP ^1])  $ 
for which the markings in $A_1$ are over $0$ and the markings in $A_2$ are over $\infty$.

The vertical arrows are the natural morphisms obtained by forgetting the mapping data to $\X$.

The two right horizontal arrows are the inclusions of the maximal $\CC ^*$-fixed closed substacks 
into their ambient stacks.

The two left horizontal arrows $\Psi _1, \Psi _2$ are obtained from gluing morphisms, constructed as
follows.   By Lemma \ref{simple}, an object in $Q(A_1, A_2, \beta)$ can be considered as a triple  
\[ ((C_1, \bx _{A_1}, x_{\bullet}), [u]_1), ((C_0, x_{\check{\bullet}}, x_{\bullet}), [u]_0), ((C_1, x_{\check{\bullet}}, \bx _{A_2}), [u]_2)\]
of $\ke$-stable quasimaps to $X$, $\ke$-stable graph quasimaps to $X$, $\ke$-stable quasimaps to $X$, respectively,
 with isomorphisms $x_{\bullet}\cong x_{\check{\bullet}}$ inverting the band structures (for each \lq\lq join"). With these isomorphisms  
we glue $C_1, C_0, C_2$ along $x_{\bullet}, x_{\check{\bullet}}$ to get a twisted curve over $S$:
\[ C_1\  {}_{x_{\bullet}}\!\!\amalg_{x _{\check{\bullet}}}  C_0\  {}_{x _{\bullet}}\!\!\amalg _{x _{\check{\bullet}}} C_2 \] 
(see \cite[\S A]{AGV} for gluing of Artin stacks along closed substacks). By the co-cartesian property
the mapping data $[u]_1: C_1\ra \X$ (followed by  the inclusion $\X \cong \X \times [0,1]\subset \X \times \PP ^1$), 
$[u]_0: C_0 \ra X\times \PP ^1$, $[u]_2: C_2\ra \X$ (followed by the inclusion  $\X \cong \X \times [1,0]\subset \X \times \PP ^1$) 
can be glued. This explains $\Psi_1$. The construction of $\Psi _2$ is similar.
By investigating the formal deformation spaces we note that 
$\Psi _2$ is an 1-isomorphism. So is $\Psi _1$, since the left square in the diagram is cartesian.

\begin{Prop}\label{comp_vir} 
Let \[ \Delta _2: \bIX \times \bIX \ra (\bIX \times \bIX ) \times (\bIX \times \bIX) \] be the product of the diagonal map $(\mathrm{id}, \iota): \bIX \ra \bIX \times \bIX$. Then
 \begin{align*} & [F_{A_1, A_2, \beta} ] ^{\vir}   = \sum _{A_1\coprod A_2 = [k], \beta _1 + \beta _2 = \beta}  \\ &
 \Delta _2 ^! ( [Q_{0, A_1\cup \{\bullet\}}^\ke  (X, \beta _1)]^{\vir} \times [\IX ] \times  [Q_{0, A_2\cup \{\check{\bullet}\} }^\ke  (X, \beta _2)]^{\vir} ) . \end{align*}
\end{Prop}

\begin{proof} First note that $ \fM _{0, k} ^{\tw} (\PP ^1, [\PP ^1]) $ is smooth over $\fM _{0, k}^{\tw}$ so that 
there is a canonical quasi-isomorphism 
\[ (\bL_{ \fM _{0, k} ^{\tw} (\PP ^1, [\PP ^1])}   |_{ \fM _{A_1, A_2} ^{\tw, \CC ^* } }  ) ^{\CC ^*\text{-fixed}} \cong \bL_{ \fM _{A_1, A_2} ^{\tw, \CC ^*} }.\]
Therefore, if we denote by $E$ the absolute perfect obstruction theory for $QG^\ke _\beta$, then the $\CC ^*$-fixed part of the distinguished triangle 
\[ E _{|_{ F_{A_1, A_2, \beta}}} \ra (R^\bullet \pi _* [u_1]^* \mathbb{T}_{\X} )^\vee \ra \bL_{ \fM _{0, k} ^{\tw} (\PP ^1, [\PP ^1])}    |_{ F_{A_1, A_2, \beta} }[1] \]
yields the compatible absolute and relative perfect obstruction theories.
By the normalization sequence of nodal curves with distinguished nodes and the tangent bundle lemma \cite[Lemma 3.6.1]{AGV}, we can apply 
 the functoriality from  \cite[Proposition 7.5]{BF} to complete the proof.
\end{proof}

\bigskip

\noindent{\em Proof of Proposition \ref{factorization_Prop}}.
We compute the Euler class of the virtual normal bundle of $F_{A_1, A_2, \beta}$ in $QG_\beta ^\ke$ when $0\in A_1$, $\infty \in A_2$. 
Assume $\beta _i \ne 0$, $i=1, 2$. 
Then the Euler class comes from two contributions:
\begin{enumerate}
\item the moving part of the deformation of $u_2: C \ra \PP ^1$; 
\item  the moving part of
the deformation and the infinitesimal automorphism of $(C, \bx )$.
\end{enumerate}

Part (1) contributes $z$ (respectively, contributes $-z$) for each connected component of $C\setminus C_0$ contracted to $0$ (respectively, contracted to $\infty$) 
under $[u]_2$.

Part (2) has contributions exactly from $T_{\bar{p}_i}C_1\ot T_{\bar{p}_{0,i}} C_0$ for each separating node $\bar{p}_i$ of $C$ (here $\bar{p}_i$ are
the geometric points associated to the separating node).
The contributions are $\frac{z-\psi _{\bullet}}{r}$ for $\bar{p}_1$  and $\frac{-z-\psi _{\check{\bullet}}}{r}$ for $\bar{p}_2$ if $r=|\mathrm{Aut}(\bar{p}_i)|$.

When $\beta _1=0$, $|A_1|=1$ and $\beta _2\ne 0$ or $|A_2| > 1$ there is no contribution $\frac{z-\psi _{\bullet}}{r}$; this is exactly the case when
the separating node $\bar{p}_1$ disappears.
Similarly, when $\beta _2= 0$, $|A_2|=1$ and $\beta _1\ne 0$ or $|A_1|>1$, then
the separating node $\bar{p}_i$, respectively $i=1, 2$ disappears and
there is no contribution  $\frac{-z-\psi _{\check{\bullet}}}{r}$.
Finally, when $\beta _1=\beta _2 =0$ and $|A_1|=|A_2|=1$ 
both separating nodes disappear and there is neither contribution $\frac{z-\psi _{\bullet}}{r}$  nor $\frac{-z-\psi _{\check{\bullet}}}{r}$.

The factorization expression \eqref{fact} is immediate by this analysis, Proposition \ref{comp_vir}, and the contributions of $\eta _0$ and $\eta _\infty$. $\Box$

\begin{Lemma}\label{J_form} Fix $\ke \in \QQ_{>0}$. The $J^\ke$-function of $X=[W^{ss}/G]$ takes the form
\begin{align*}
J^\ke({t}, z) 
\nonumber&=& \one _{X}+ \frac{{t}}{z} + \sum _{0 < \beta(L_\theta)\leq 1/\ke}
q^\beta (\widetilde{\pr _{\bIX} \circ ev_\star})_*
\frac{(-z)[F^{\emptyset, \beta}_{\star, 0}]^{\mathrm{vir}}}{\mathrm{e}^{\CC^*\times T}(N^{\vir}_{F^{\emptyset , \beta}_{\star , 0}/ QG ^\ke_{0, \{ \star \}, \beta }(X) })}\\
\nonumber&+& \sum_{ {(\beta \ne 0, k\geq 1)\;{ \mathrm{ or }}} \atop{(\beta(L_\theta)> 1/\ke, k=0)} }q^\beta 
(\ev_\bullet)_*\frac{\prod_{i=1}^k ev_i^*({ t})\cap[Q_{0, [k]\cup\{ \bullet\}}^\ke (X, \beta )]^{\mathrm{vir}}}{k!z(z-\psi_\bullet)},
\end{align*}
where $\one _{X}$ is the fundamental class of $X$ on the untwisted sector.
Here again the localization residue is taken as a sum over the 
connected components of $F^{\emptyset, \beta}_{\star , 0}$.
\end{Lemma}
\begin{proof}
The first term is from the case $(k=0, \beta =0)$. In this case, $r=1$ by Lemma \ref{simple} (1). Thus,
$(F_{\star, 0}^{0, 0}, QG^{\ke}_{0, \star , 0} (X))= (X\times \infty, X\times \PP ^1)$. This explains the first term.

The second term is from the case $(k=1, \beta=0)$. By Lemma \ref{simple} (2), 
$ev _\star: F_{\star, 0}^{1, 0}\ra \bIX$ is nothing but  $\varpi : \IX \ra \bIX$ under a suitable identification of
$F_{\star, 0}^{1, 0} = \IX$.  Since the degree of $I_{\mu _r}X\ra \bar{I}_{\mu _r}X$  is $1/r$, we obtain the second term.

The third sum is from the case  $k= 0$ and  $0\ne \beta (L_\theta ) \le 1/\ke$. 

The last sum follows from the case when there is a separating node over $[0,1]\in \PP ^1$, i.e., 
$k\ge 1$ or $\beta > 1/\ke$ except $(k, \beta)=(1,0)$.
By Proposition \ref{comp_vir}  with one separating node over $[0,1]$
and the analysis of the virtual normal bundle in the proof \eqref{fact}, the last sum is immediate. 
\end{proof}

Lemma \ref{J_form} shows that $J^\infty$ coincides with the $J$-function $J$  defined in \cite{Tseng} after identifying
$\HA ^*(\IX)$ and $\HA ^*(\bIX)$ by pullback (not by push-forward):
$\varpi ^* (z J^\infty) =  J$,
where $\varpi$ denotes the rigidification $\IX\ra\bIX$.

\subsection{Unitarity of $S^\ke$}
The following Proposition says that the operators $S^\ke_t$ are symplectic transformations on Givental's symplectic space $\cH_X$.

\begin{Prop}\label{Unitarity}
Consider the operator
\[ (S^\ke)^\star _t (-z) (\gamma ) := \sum _i  \gamma _i \lla \gamma ^i , \frac{\gamma }{-z-\psi} \rra ^{\ke}_{0, \{\star , \bullet\}} = \gamma + O(1/z) . \]
Then \[ (S^\ke)^\star (-z) \circ S^\ke (z) (\gamma ) = \gamma . \]
\end{Prop}
\begin{proof}
The factorization \eqref{fact},  together with the fact that $S^\ke$ and $(S^\ke)^\star$ are of form $\mathrm{Id} + O (1/z)$,
yields a proof as in \cite[Proposition 5.3.1]{CKg0}.
\end{proof}

\subsection{The $P$-series and Birkhoff factorization of $J^\ke$}


Define the $P^\ke$-series by
\[ P^\ke (t, q, z) := \sum \gamma _i \lla \gamma ^i \eta_\infty \rra _{\{\star\}}^{QG^\ke} = \one_X + O(q),\]
where the latter equality follows from Lemma \ref{simple}.

\begin{Def}\label{semi-pos}
We call the triple $(W, G, \theta)$ (or the \lq\lq target" $(\X, X)$)  {\em semi-positive} if $\beta (\det T_{\X})$ is non-negative
for every $\beta\in\mathrm{Eff}(W,\G,\theta)$. (Note that $\beta (\det T_{[W/G]})=\beta (\det T_W)$, since the determinant of adjoint bundle $P\times_\G \mathrm{Lie}(\G)$ 
on the domain curve $C$ has degree zero.)
\end{Def}

\begin{Thm}\label{JS}
\begin{enumerate}

\item The following formula holds: \[ J^{\ke} (t, q, z)   = S^\ke _t (z) (P^\ke (t, q, z) ) . \]

\item\label{grading}  The small $J$-function $J^\ke |_{t=0}$ has degree $0$ if we set $\deg z = 1$, 
$\deg q^\beta = \beta (\det T_\X)$, and the degrees of cohomology classes of
$\bIX$ to be the age-shifted complex degrees. 

\item 
For a semi-positive target $(W, G, \theta )$, the following hold.

\begin{enumerate}

\item $J^{\ke} (t, q, z) $ takes
the form $J^\ke _0(q)\one _X  +  \frac{1}{z} (t+ J^\ke _1(q)) + O (1/z^2)$,
 for some $J^\ke _0(q) = 1 + O(q) \in \Lambda _{K}$ of degree $0$ and  $J^\ke_1(q)   \in (q\Lambda _{K})H^{\leq 2}(\bIX,\Lambda _{K})$ of degree $1$
 with respect to the grading from \eqref{grading}  above.

\item  $P^{\ke}(t, q, z) = J^{\ke}_0(q) \one _X$.  In particular
\begin{align*}  \frac{J^\ke (t, q, z)}{J^\ke_0 (q) } & =  S^{\ke} _t ( \one _X ), \\
   \frac{t + J_1^\ke (q)}{J_0^\ke (q)} & = \sum \gamma _i \lla \gamma ^i, \one _X \rra _{0,2}^\ke -\one _X . \end{align*}

\item  
$J_0^\ke (q) \one _X $ is the unity in the $\ke$-quasimap quantum product.

\end{enumerate}

\item  If $\beta (\det T_{\X}) > 0$ for every nonzero $L_\theta$-effective class, then $J^\ke =\one _X + O (1/z)$.

\end{enumerate}

  \end{Thm}

\begin{proof} The proof of (1) is completely analogous  to the case when $X$ is a nonsingular variety (see 
\cite[Theorem 5.4.1]{CKg0}). Briefly, virtual localization of $P^\ke$ with respect to the $\CC^*$ action
on graph spaces yields the factorization 
\begin{equation}\label{P from J} P^\ke = (S^\ke)^*(-z) (J^\ke) .\end{equation}
Now apply $S^\ke(z)$ to both sides of the above equality and use Proposition \ref{Unitarity}.

For (2), note that the virtual dimension of a connected component of $QG_{0, 1, \beta} ^\ke$ landing on $X_c$ under
$\pr _{\bIX} \circ ev_{\star}$ is, by the formula \eqref{dim},
$$ 1+ (\dim \X + 1 - 3) + \beta (\det T_{\X}) + \deg [u]_2^* T_\PP ^1 - \mathrm{age} X_c ,$$  so that
under $(\pr _{\bIX}\circ \ev _\star )_*$ after cap with $\eta_\infty$, its usual cohomological degree becomes 
$-\beta (\det T_{\X})  - \dim \X + \mathrm{age} X_c +  \dim X_c$. The age shifted degree of the latter
becomes $-\beta (\det T_{\X})  - \dim \X + \mathrm{age} X_c +  \dim X_c  + \mathrm{age} \iota (X_c)$  which is $-\beta (\det T_{\X})$.

For (3a), note that by (2) and Lemma \ref{J_form}, $J^\ke$ is a series of $1/z$. Considering the equality of (1) modulo $1/z$, we conclude that
$P^\ke$ has only the zero-th power of $z$ and coincides with  the zero-th power of $z$ piece  in $J^\ke$.
The latter is  $J_0^\ke (q) \one_X$ by (2) for some $J_0^{\ke}(q) \in  \Lambda _{K}$.

Claim (3b) follows from (1) and (3a) (see \cite[Corollary 5.5.3]{CKg0}).
Claim (3c) follows  from the identical proof of \cite[Corollary 5.5.4]{CKg0}.

Claim (4) is obvious from (2) and Lemma \ref{J_form}.
\end{proof}

\begin{Rmk}
When $\ke = \infty$, combining the relation \eqref{P from J} from the proof of Theorem \ref{JS} (1) with Lemma \ref{J_form}, we obtain  \[ J^\infty = S^\infty (\one _X ). \]
\end{Rmk}

\subsection{$\ke$-wall-crossing}

As generalizations of Theorems 7.3.1, 7.3.4 and Conjecture 6.2.1 of  \cite{CKg0}, it is natural to make the following conjecture.

\begin{Conj}\label{Main_Conj} 
\begin{enumerate}
\item\label{Part1} The following formula holds: \[ S_t^{\ke} (\one _X) = S_{\tau (t)} ^\infty (\one _X) 
\text{ for }
\tau (t):= t + \sum _{\beta \ne 0} q^\beta \gamma _i \lla \gamma ^i , \one_X \rra _{0, \{ \bullet , \star \} , \beta}^{\ke} . \]

\item\label{Part2}  The following formula holds:  \[ J^{\ke}(t, q, z)= S_t ^{\ke}(z) (P^\ke (t, q, z)) = S_{\tau ^{\infty, \ke}(t)} ^\infty (z) (P^{\infty, \ke} ( \tau ^{\infty, \ke}(t) , q,  z ) ) \]
for a unique transformation 
$$t\mapsto \tau ^{\infty, \ke}(t) \in \HA_T^*(\bIX)\otimes_{\QQ} K [[\Eff (W, G, \theta ) ]][[\{t_i\}_i]]$$ and a unique element 
$$P^{\infty, \ke} (t, q,  z) \in \HA_T^*(\bIX)\otimes_{\QQ} K[z] [[\Eff (W, G, \theta ) ]][[\{t_i\}_i]] .$$
In particular, $zJ^{\ke}$ is on the Lagrangian cone of the Gromov-Witten theory of $X$ defined in
\cite{Tseng} (with Novikov variables from $\Lambda$) and if the triple  $(W, G, \theta)$ is  semi-positive, then
 \[ J^{\ke}(t, q,z) / J^{\ke}_0(q) = J^{\infty}( (t + J^{\ke}_1(q))/J^{\ke}_0(q), q,z) \]
 with $\tau ^{\infty, \ke}(t ) = (t + J^{\ke}_1(q))/J^{\ke}_0(q)$ and $P ^{\infty, \ke}(t, q, z) = J_0^{\ke}(q) \one _X$.
\end{enumerate}
\end{Conj}

The first main result of the paper is a proof of Conjecture \ref{Main_Conj} in the presence of a torus action with good properties.

\begin{Thm}\label{Comp_Thm}
Suppose that  there is an action by an algebraic torus $T$ on $W$ which commutes with the $G$ action and such that the induced
$T$ action on the coarse moduli space $\uX$ of $X$ has only isolated fixed points. Then Conjecture \ref{Main_Conj} \eqref{Part1} holds true.

Further, Conjecture \ref{Main_Conj} \eqref{Part2} holds true if we assume in addition that 
 the 1-dimensional $T$-orbits are isolated when $(W, G, \theta)$ is not semi-positive.
\end{Thm}

\subsection{Proof of Theorem \ref{Comp_Thm}}\label{Pf_Comp_Thm}

For $X$ a scheme, the result is contained in \cite[Theorems 7.3.1, 7.3.4]{CKg0} and the proof given there also works for orbifolds.
We outline the argument, focusing on the appropriate changes.

\subsubsection{Unbroken quasimaps} (cf. \cite[\S 7.4.3]{CKg0})
For each $\CC$-element in $\Qket ^T$, we say the element 
is of {\em initial type} (resp. of {\em recursion type})  if the cotangent $T$-weight of the first marking at 
the coarse domain curve is zero (resp. nonzero). 
A recursion element is called 
{\em unbroken}  if
$$\ka _{C', x} + \ka _{C'', x} = 0$$
for every node $x$ of the domain curve $C$ connecting $C'$ and $C''$,
 where 
$\ka _{C', x}$, $\ka _{C'', x}$ are the induced $T$-weights of the cotangent spaces
at $x$ to  the coarse curves $\uC' $ and $\uC''$. A recursion element is called {\em broken} if
there is a node $x$ of the domain curve $C$ connecting $C'$ and $C''$ with $\ka _{C', x} + \ka _{C'', x} \ne 0$.
Note that the domain curve of an unbroken recursion element has no components 
contracted under $[u]_{\reg}$. 

Let $M$ be a connected component of $Q ^{\ke} _{0, 2}(X, \beta )^T$. We call
$M$ a {\em recursion component}  of  $Q ^{\ke} _{0, 2}(X, \beta )^T$ if it contains 
an unbroken element. This implies that every element in $M$ is unbroken and is therefore a two-pointed {\em stable map to $X$}. Hence
$M$ is canonically identified with a connected component of $\cK _{0,2}(X, \beta )^T$.

\subsubsection{Recursion}\label{recursion_subsection} (cf. \cite[\S 7.5]{CKg0})
We divide the connected components of the $T$-fixed substack $$\Qket ^T$$ into initial types and recursion types, according to 
whether the first marking is on a contracted component or not under $[u]_{\reg}$. 
Every recursion component is of the form $M\cong M'\times _{ \bIX ^T} M''$,
with $M'$ an unbroken component of $\cK _{0, 2}(X, \beta ')$ for some $\beta '$ 
and $M''$ a connected component of $Q^{\ke}_{0,2+m} (X, \beta - \beta ')$,
such that $\ka _{M}=\ka _{M'} \ne  \ka _{M''}$. Here $\ka _M$, $\ka _{M'}$, $\ka _{M''}$ denote the respective cotangent 
weights at the first coarse markings of any element in $M$, $M'$, $M''$.

Let $$\{\nu, \nu ', ...\} = \underline{\bIX ^T}$$ be  the finite set of  $T$-fixed $\CC$-points of $\bIX$ and let
$$i_\nu: \nu = \Spec \CC \ra \bIX$$ be the associated map.  For $\gamma\in H^*_{T}(\bIX)\ot\Lambda_K$, denote 
 \begin{equation}\label{Snu}  S_\nu^\ke (\gamma) = i_\nu^*(S_{t}^\ke(z)(\gamma)) \in  K [[1/z]][[\mathrm{Eff}(W,G,\theta)]] [[\{t _j\}_j ]]
 . \end{equation}

Calculating $S_\nu^\ke(\gamma)$ by virtual localization and using
the above analysis of $\Qket ^T$ to separate the contributions from components of initial type and the contributions from components of recursion type, 
we conclude the following Lemmas, whose proofs are word for word the same as the proofs of Lemma 7.5.1 and of the Recursion Lemma 7.5.2 in \cite{CKg0}. 

\begin{Lemma} Each $(q,\{t_i\})$-coefficient of $S_\nu^\ke$ is naturally an element of $K(z)$. Further, this rational function decomposes as a sum of partial simple fractions with denominators
either powers of $z$, or powers of $(z-\alpha)$ with $-n\alpha$ a weight of the $T$-representation on the tangent space $T_\nu\bIX$ for some $n\in\QQ_{>0}$.

\end{Lemma}

\begin{Lemma} \label{recursion lemma}
 $S_\nu^\ke$ satisfies the recursion relation
       \begin{equation}\label{recursion formula} 
       S_\nu ^\ke (z) = R_\nu^\ke (z) + \sum _{ M'\in U(\nu)} q^{\beta_{M'}} 
\left\lan \frac{\delta _\nu}{z - \ka _{M'}- \psi _0 }, S_{\nu '_{M'}, \ka _{M'}}^\ke |_{z=\ka _{M'}-\psi _{\infty}}
\right\ran _{M'}
       \end{equation} 
where: 
\begin{itemize}

\item $\delta _\nu  : = (i _\nu) _* (\one _{\nu}) $. 
\item $U(\nu)$ is the set of all unbroken components $M'\subset \cK _{0,2}(X, \beta _{M'})^T$ for varying $\beta _{M'}$ 
for which the first marking lands on $\nu$ via the evaluation map. 
\item $\nu '_{M'} \in (\underline{\bIX} )^T$ denotes $\underline{f (x_2)}$ for $((C, x_1, x_2), f) \in M'$.
\item $\psi _0$, $\psi _\infty$ are the nonequivariant Psi classes on $M'$ associated to the markings $x_1$, $x_2$, respectively.
\item $R_\nu ^\ke (z)$  is the contribution from all components of initial type and has the property that each of its $(q,\{t_i\})$-coefficients is an element in $K[1/z]$.
\item For $\nu \in (\underline{\bIX})^{\T}$, $S_{\nu, \ka}^\ke (z)$ denotes the part of $S_\nu^\ke (z)$ remaining after the partial fraction terms with poles at $z=\ka$ are removed.
\item The subscript $M'$ for the bracket means the $\T$-virtual localization contribution of the component $M'$ to the virtual intersection 
number on $\cK_{0,2}(X, \beta _{M'} )$.
\end{itemize}

\end{Lemma}

\subsubsection{Uniqueness} 

Lemmas parallel to  Polynomiality Lemma 7.6.1, Uniqueness Lemma 7.7.1  of  \cite{CKg0} hold.
We mention the needed minor modifications. 
\begin{enumerate} 

\item Instead of $X^T$, we need to use the set of $\CC$-points of $\bIX ^T$, i.e., 
\[ \{ (x, g) : x\in X^T (\CC), g \in \Aut (x) \} .\]

\item In the polynomiality formula in Lemma 7.6.1 of \cite{CKg0}, the product between two $S^\ke_\nu$ now becomes
 \[ S^\ke_\nu (q, t, z) \iota ^* (S^\ke _{\nu ^{-1}} (qe^{-zyL_\theta}, t, -z))  , \] 
 where $\nu ^{-1}$ denotes the point obtained from $\nu$ after inverting the band structure, i.e.,
 $\nu ^{-1}:= (x, g^{-1})$ for $\nu = (x, g)$.
  Note that $\iota ^*S^{\ke}_{\nu ^{-1}} (q, t, z) = i^*_{\nu} S ^{\ke}(q, t, z) = S^{\ke} _{\nu } (q, t ,z )$.

\item $S^{\ke}_\nu (q, t , z)$ modulo $q$ does not depend on $\ke$. (This is observed in Remark \ref{mod q}(3) and replaces 
 condition (5) in Uniqueness Lemma 7.7.1 of \cite{CKg0}.)
 
 \end{enumerate}
 
With these changes in mind, the argument of \cite[\S 7.8, 7.9]{CKg0} applies to the orbifold setting and provides
the proof of Theorem \ref{Comp_Thm}.

\begin{Rmk} {\it (Twisted theories)}\label{twisted}
 Let $\cE$ be a vector bundle on $\X$ or, equivalently, a $\G$-equivariant vector bundle on $W$ (for example, 
 one could take $\cE=W\times E$,
with $E$ a $\G$-representation). Given an invertible $\mathbb{G}_m$-equivariant multiplicative characteristic class $c$, we may define $(\cE,c)$-twisted orbifold $\ke$-quasimap invariants exactly as in 
\cite[\S 6.2]{CKM} (here $\mathbb{G}_m=\CC^*$ acts by scaling in the fibers of vector bundles). Let $c=Euler$ be the $\mathbb{G}_m$-equivariant Euler class. We then have the twisted versions of the $J^\ke$-functions and of the $S^\ke$-operators, see \cite[\S7.2]{CKg0}. 


 Let $s\in \Gamma(W,\cE)^{\G}$ be a regular section and let $Z\subset W$ be its $\G$-invariant zero locus. Assume that $Z^{ss}=W^{ss}\cap Z$ is nonsingular.
The bundle $\cE$ descends to a vector bundle $\ucE$ on $X$ and $s$ descends to a regular section $\underline{s}$ of $\ucE$. The pair $(Y:=[Z^{ss}/\G], \fY:=[Z/\G])$ (or the triple $(Z,\G,\theta)$)
satisfies the conditions of Theorem \ref{Found_Thm} and therefore has its own quasimap theory.
Note that $Y$ is the zero locus of $\underline{s}$.

If $\cE$ is convex (by definition, this means that $H^1(C,[u]^*\cE)=0$
for all genus zero $\theta$-quasimaps $[u]:C\lra\X$), then in genus zero and after setting to zero the equivariant parameter for the $\mathbb{G}_m$-action, the $(\cE, Euler)$-twisted quasimap theory of $X$ coincides with the quasimap theory of $Y$ with insertions restricted to Chen-Ruan cohomology classes pulled-back
from $\bIX$. This is a consequence of the fact that \cite[Proposition 6.2.2]{CKM} holds equally in the orbifold case (the same argument, based on \cite{KKP}, works, see also \cite[Proposition 5.1]{Coates et al}).

As in \cite{CKg0}, the proof of Theorem \ref{Comp_Thm} applies also for the $(\cE, Euler)$-twisted theories,
essentially because all splitting properties of the virtual classes $[\Qket]^{\mathrm{vir}}$ required for the Recursion and Polynomiality Lemmas continue to hold after twisting. 
Hence in this situation Theorem \ref{Comp_Thm} gives the wall-crossing formulas for $Y$ under the same assumptions on the $T$-action on $W$.

Note that, as explained in \cite{Coates et al}, convexity is a rather restrictive condition on orbifold targets. 
For example, in the case of a split $\cE=\oplus L_{\eta_i}$ (i.e., $Y$ is a complete intersection in $X$) the positivity condition of \cite[Proposition 6.2.3(i)]{CKM} does not suffice to guarantee convexity and
has to be supplemented with the requirement that the line bundles induced by $L_{\eta_i}$ on $X$ are pulled-back from the coarse moduli $\uX$. 

\end{Rmk}

\section{$\mathds{I}$-functions and stacky loop spaces}

In this section we introduce (after \cite{BigI}) a generalization of Givental's small $I$-function and prove that it lies on the Lagrangian cone of the Gromov-Witten theory of $X$. 
We then show that by using certain \lq\lq stacky loop spaces" of
$\theta$-quasimaps to $\X$, these new $\mathds{I}$-functions can be explicitly computed.

\subsection{ $\mathds{I}$-functions}

Denote by $I (t, q, z)$ the $J$-function for $(0+)$-quasimap theory $J^{0+}(t, q, z)$ and let 
\[ I(0, q, z) := \sum q^{\beta} I_{\beta}(0, q, z) \] 
be its specialization at $t=0$.

Let $t=\sum _i t_i \gamma _i \in H^*_T(X)$ (in the untwisted sector), with the sum taken only over those  $\gamma _i$ 
which can be written as a polynomial 
\begin{equation}\label{poly_exp}p_i( c_1 (L_{\eta_{ij} }))=p_i( c_1 (L_{\eta_{i1} }),\dots , c_1 (L_{\eta_{im_i} }))\end{equation}
in divisor classes of the form $c_1(L_{\eta_{ij}})$ for some  $\eta_{ij} \in \chi(G)$. 

\begin{Def}\label{new}
\[ \mathds{I} (t, q, z)  = \sum  _{\beta} q^{\beta} \exp (\frac{1}{z}\sum _i t_i p_i( c_1 (L_{\eta_{ij} })+\beta(L_{\eta_{ij} })z )  I_\beta (0, q, z). \] 
\end{Def}

The specialization of $\mathds{I}(t, q, z)$ to $$t=t_0\one _X + \sum t_i c_1(L_{\eta _i})  \in H^{\le 2}_T(X)$$ 
for $\eta _i \in \chi (G)$
is called {\em Givental's small $I$-function}  and denoted by
 $I^{\Giv} (t, q, z)$.

\medskip

The following is the second main result of the paper.

\begin{Thm}\label{big_I}
Assume that the $T$-action on $\uX$ has isolated $T$-fixed points and isolated
$1$-dimensional $T$-orbits. Then  $\mathds{I}(t, q, z)$ is on the Lagrangian cone of the Gromov-Witten theory of $X$.
\end{Thm}

\begin{proof} The argument is similar to the proof (under the same assumptions on the $T$-action) of the second part of Theorem \ref{Comp_Thm}.

By Lemma 6.4.1 of \cite{CKg0}, there is a unique transformation
$$t\mapsto \tau ^{\infty, \ke}(t) \in \HA_T^*(\bIX)\otimes_{\QQ} K [[\Eff (W, G, \theta ) ]][[\{t_i\}_i]]$$ and a unique element 
$$P^{\infty, \ke} (t, q,  z) \in \HA_T^*(\bIX)\otimes_{\QQ} K[z] [[\Eff (W, G, \theta ) ]][[\{t_i\}_i]]$$
such that $$\mathds{I}  (t, q, z) = S ^{\infty}_{\tau ^{\infty, s} (t)} (P ^{\infty, s}(t, q, z)) \text{ modulo } 1/z^2 . $$
 Now if we let  \begin{align*} S_{1, \nu} & = i_{\nu}^* \mathds{I}(t, q, z) , \\
  S_{2, \nu} & =  i_{\nu}^* S ^{\infty}_{\tau ^{\infty, s}  (t)} (P^{\infty, s} (t, q, z)) , \end{align*}
then it is straightforward to check that the systems $\{ S_{i, \nu} \ | \ \nu \in \underline{\bIX }^T \}$, $i=1,2$  satisfy all the properties in
 Uniqueness Lemma 7.7.1 of  \cite{CKg0} (replacing $\WmodG ^T$ by $\underline{\bIX}^T$ and condition (5) of Uniqueness Lemma 7.7.1 \cite{CKg0}
 by $S_{1, \nu} = S_{2, \nu}$ modulo $q$) so that
 $S_{1, \nu} = S_{2, \nu}$ for all $\nu$. 
 More precisely, we need to check the conditions (1) -- (3) of the Uniqueness Lemma.
 For $\{S_{2,\nu}\}_\nu$ this is already done in \S \ref{Pf_Comp_Thm}. Similarly, \S \ref{Pf_Comp_Thm} shows
 that $\{ i_\nu^* I(0, q, z) \} _{\nu} $ satisfies conditions (1) -- (3) (since it is the specialization of 
 $\{ i_\nu ^* S_t ^{0+} (P ^{0+}(t, q, z)) \}_{\nu}$ at $t=0$). From this fact and the explicit form of the 
 exponential correcting factor, by using the simple observation
 \[ (c_1(L_{\eta _{ij}})  + \beta (L_{\eta_{ij}} ) z )|_{\nu, z=\ka _{M'}} = (c_1(L_{\eta _{ij}})  + (\beta - \beta _{M'}) (L_{\eta_{ij}} ) z )|_{\nu '_{M'}, z=\ka _{M'}},
 \] 
 it follows by a direct check that
 conditions (1) -- (3) continue to hold for $\{ i_\nu ^*\mathds{I} (t, q, z)  \}_{\nu}$.
\end{proof}

\begin{Rmk} Alternatively, and better, one may use the geometric arguments of 
\S \ref{Pf_Comp_Thm} to show that $\{ i_\nu^* \mathds{I}(t, q, z) \} _{\nu} $ satisfies conditions (1) -- (3) in the
Uniqueness Lemma 7.7.1 of  \cite{CKg0}. We explain this briefly. Consider the quasimap theory with weighted 
markings from \cite{BigI}, extended to orbifold targets using \S \ref{weighted}.
Fix the stability $(0+, 0+)$, that is, the asymptotic stability with respect to $\theta$ and infinitesimally small weights
on the weighted markings. For this theory, we have the $J$-function
$\mathds{J}^{0+, 0+}({\bf t}, q, z)$, the $S$-operator $\mathds{S}_{\bf t} ^{0+, 0+}$, and the $P$-series 
$P^{0+, 0+} ({\bf t}, q, z)$, satisfying the Birkhoff factorization 
\begin{align}\label{NewBir} \mathds{J}^{0+, 0+} ({\bf t}, q, z) = \mathds{S}_{\bf t}^{0+, 0+} (P ^{0+, 0+} ({\bf t}, q, z)) .\end{align}
Here ${\bf t}\in H^*_{T}([W/G])$ is a general element.
By \S 5 of \cite{BigI}, the new $I$-function $\mathds{I}$ of Definition \ref{new} is identical with 
$\mathds{J}^{0+, 0+}({\bf t}, q, z)$ after the specialization ${\bf t }= \sum t_i \tilde{\gamma _i}$ where
$\tilde{\gamma _i}$ is the natural lift of $\gamma _i = p_i( c_1 (L_{\eta_{ij} }))$ by taking
the Chern class of the corresponding line bundles on $[W/G]$.  
Through this identification, the fact that  $\{ i_\nu^* \mathds{I}(t, q, z) \} _{\nu} $ satisfies conditions (1) -- (3)
follows from the geometric argument of \S \ref{Pf_Comp_Thm}  applied to the right hand side of \eqref{NewBir}.
\end{Rmk}

\begin{Rmk} Besides the overlap with the Mirror Theorem for toric DM stacks proved in \cite{CCIT} 
(which will be explained in the next section, see Corollary 5.3.4 (3) and the discussion after it), Theorem 4.1.2 also overlaps with the work of C. Woodward, \cite{Wo1, Wo3}.

Partly in collaboration with E. Gonzales, Woodward has investigated the so-called gauged maps from curves to certain $\G$-varieties. We comment briefly on the similarities and differences
with quasimap theory.
First, the theory of gauged maps requires a parametrized component in the domain
curve, so it is essentially a genus zero theory, unlike ours. The gauged maps used by Woodward are in particular representable maps 
to $[W/\G]\times\PP^1$ of class $(\beta,1)$,  as are the graph quasimaps from \S 2.5.3, or the variant with weighted markings from \S 2.5.5 of this paper. 
However, the stability conditions he considers are quite different from the ones we employ, so he obtains different compactified graph spaces. 

There are
also differences in the kind of targets allowed by the two theories: the results in \cite{Wo1, Wo3} based on the Gonzales-Woodward theory allow $W$ to be either a
smooth projective variety or a vector space, while we require $W$ affine,
but allow lci singularities.

When comparing the two theories, the moduli spaces that resemble each other most closely are our genus zero graph spaces with
$\ke=0+$ and infinitesimally weighted markings, and their compactified moduli spaces of ``large area gauged maps", but even these are not exactly the same when markings are present. Nevertheless, it appears that the ``localized gauged graph potential" $\tau^\G_{X,-}(\alpha,\hbar,q)$ from Definition 9.13 of \cite{Wo3} (in the limit $\rho\ra\infty$)
should be equal to $\mathbb{J}^{0+,0+}({\bf t}=\alpha,z=\hbar,q)$. Therefore Theorem 1.6 of \cite{Wo1}, 
whose proof is now contained in \cite{Wo3},
and our Theorem 4.1.2 should give the same result for the set of targets they both cover.

Note, however, that there is an incompatibility in the sample calculation for toric manifolds in \cite[Example 9.1.5]{Wo3}: the equality
$\tau^\G_{X,-}(\alpha,\hbar,q)=\exp(\alpha/\hbar)\tau^{\G,0}_{X,-}(\hbar,q)$ is not compatible with Theorem 1.6 of \cite{Wo1,Wo3} since the right-hand side 
does {\it not} satisfy Givental's recursion and therefore cannot lie on the Lagrangian cone of the Gromov-Witten theory of $X$, while Woodward's Theorem 1.6 states that the left-hand side {\it is} on the Lagrangian cone.

\end{Rmk}


 \subsection{Stacky loop spaces}\label{Stacky_Loop}

In view of (the semi-positive case of) Theorem \ref{Comp_Thm} and Theorem \ref{big_I}, it is important to compute explicitly the function $I(0,q,z)=J^{0+}|_{t=0}$. For this purpose,
we construct another quasimap graph space.
When $X$ is a weighted projective space, this construction already appeared in \cite{CCLT}.

Denote by $\PP _{a,1}$ the quotient stack $[\CC ^2\setminus \{0\}/\CC ^*]$ where $\CC^*$ acts on $\CC ^2$ by
weights $a$ and $1$. 
Note that $0:=[0,1]$ is a schematic point, while $\infty:=[1,0]\cong B\mu _a$ is a stacky point for $a > 1$.

For a positive integer $a$ and $\beta \in \Hom _{\ZZ}(\Pic (\X), \QQ )$, 
define $$\Qmaprep \subset \Hom^{\mathrm{rep}}_{\beta} (\PP _{a, 1}, \X)$$ 
to be the moduli stack of all $\theta$-quasimaps $[u]$ from $\PP _{a, 1}$ to $\X=[W/G]$. This means that 
$[u]\in \Qmaprep$ is a representable morphism to $\X$, mapping the generic point of $\PP_{a,1}$ into $X$. However, the stacky point $\infty=[1,0]$ is allowed to be mapped to the unstable locus $\X \setminus X$. 

\begin{Prop}
The stack $\Qmaprep$ is a DM stack proper over $\uX _0$, equipped with
a canonical perfect obstruction theory $R^\bullet \pi _* [u]^* \mathbb{T}_{\X}^\vee$.
\end{Prop}

\begin{proof}
Since it is a substack of $\Hom _{\Spec\CC}(\PPa, \X)$, by Proposition 2.11 of \cite{Lieb}, it is an Artin stack of locally finite type over $\CC$.
By the boundedness Theorem 3.2.4 of \cite{CKM}, it is of finite type over $\CC$.
The deformation/obstruction theory is clearly given by $R^\bullet \pi _* [u]^* \mathbb{T}_{\X}^\vee$
and by the quasimap condition, there is no infinitesimal automorphism so that the stack is a DM stack.
The only remaining part is to show the stack is proper. 
As before, we use the valuative criterion for properness. 
For $a=1$, this is known by \cite{CKM}. For $a>1$, the argument in \cite{CKM} applied on
\'etale charts of $\PPa$ works. \end{proof}

The following lemma gives a condition on $a$ which is necessary for the non-emptiness of $\Qmaprep$.

\begin{Lemma}\label{minimal} Let $T(G)$ denote a maximal torus of $G$. 
Every morphism $[u]\in \Hom _{\beta} (\PPa, \X)$ induces a canonical homomorphism 
$\tilde{\beta} : \chi (T(G)) \ra \QQ$, well-defined  up to the Weyl group action on the character group
$\chi (T(G))$. Furthermore, 
$[u]$ is representable if and only if  $a$ is  the minimal positive integer
making $a\tilde{\beta}( \eta )\in \ZZ$ for all  $\eta \in \chi(T(G))$.
\end{Lemma}
\begin{proof}
Every $[u]: \PPa = [\CC ^2\setminus \{ 0\} /\CC ^*] \ra \X = [W/G]$
is the restriction of  a unique morphism $[\CC ^2/\CC ^*]\ra [W/G]$. 
The latter induces a morphism $B\CC ^*\ra BG$ or, equivalently, a group homomorphism $\lambda: \CC ^*\ra G$ unique up to
conjugacy classes. We may assume that $\lambda$ is factored through the inclusion $T(G)\subset G$.
Define $\tilde{\beta}\in \Hom (\chi (T(G)), \QQ)$ by
\[ \tilde{\beta}(\eta ) = \frac{\text{the exponent of } (\eta \circ \lambda ) }{a} \in \QQ . \]
Note that the map $[u]$ is representable if and only if 
$ \lambda _{|_{\mu _a}}$ is a monomorphism, 
where we identify $$\mu _a= \lan \xi \ran \subset \CC ^*,\;\; \xi := e^{2\pi\sqrt{-1}/a}.$$
 Let $k$ be  the smallest integer for which $\xi ^k\in \Ker \lambda _{|_{\mu _a}}$ and  $0 < k \le a$. This means that $k$ is the
smallest integer among $0 <k \le a$ such that
$k \tilde{\beta} (\eta)   \in \ZZ$, $\forall \eta$. Thus, $\lambda _{|_{\mu _a}}$ is a monomorphism if and only if $a$ is the
minimal positive integer making $a\tilde{\beta} (\eta) \in \ZZ$, $\forall \eta$.
\end{proof}

\begin{Rmk} Let the quasimap $[u]:\PPa\ra[W/G]$ be given by the data $(\PPa,P,u)$. By the extension of Grothendieck's theorem (see Theorem 2.4 and Theorem 2.7 in \cite{MT}), 
the principal $G$-bundle $P$
has a reduction $P_{T(G)}$ to the maximal torus $T(G)$, whose isomorphism class is unique up to the action of the Weyl group. The data $(\PPa,P_{T(G)},u)$ gives a quasimap
$$\widetilde{[u]}:\PPa\lra [W/T(G)]$$
which lifts $[u]$. It follows that, up to the Weyl group action, the associated numerical class
 $$\beta_{\widetilde{[u]}}:\Pic([W/T(G)])\ra \QQ$$ is uniquely determined by $[u]$. It is immediate to see that 
the homomorphism $\tilde{\beta}$ in Lemma \ref{minimal}
is the restriction of $\beta_{\widetilde{[u]}}$ to $\chi(T(G))$.

\end{Rmk}

For $\beta \in \mathrm{Eff}(W, G, \theta)$,
denote \[ Q_{\PPb}(X, \beta) = \coprod _{1\le a \le \e} Q_{\PP _{a, 1}}(X, \beta ) \] with the induced 
absolute perfect obstruction theory $(R^\bullet \pi _* [u]^*T_{\X})^\vee$.
Let $F_{\beta}$ be the distinguished $\CC^*$-fixed closed substack of $Q_{\PPb}(X, \beta)$ 
consisting of elements which have a single base-point of length $\beta(L_\theta)$ at the point $[0,1]$ of $\PP _{a,1}$ (i.e., the class $\beta$ is exactly supported at $[0,1]$). 

For simplicity we write $QG := QG ^{0+} _{0, \star, \beta } (X )$, $Q_{ \PPb} := Q_{\PPb}(X, \beta)$.

\begin{Lemma}\label{Nbhd} 
\begin{enumerate}
\item There is a natural isomorphism between an open neighborhood of $F^{\emptyset , \beta}_{\star ,0}$ 
in the closed substack $(\pi _{\PP ^1} \circ ev_\star )^{-1}(\infty)$ of $QG$ and an open neighborhood of
$F_{\beta} $ in $Q_{ \PPb} $, under which $F^{\emptyset , \beta}_{\star ,0} \cong F_{\beta}$.
The isomorphism preserves the $\CC^*\times T$-equivariant perfect obstruction theories. 
\item Under the natural isomorphism between $F^{\emptyset , \beta}_{\star ,0}$ and $F_{\beta} $,
\[ \frac{(ev_{\star})^*(\eta_\infty) [F^{\emptyset , \beta}_{\star ,0}]^{\vir}}{ e^{\CC ^*\times T}(N^{\vir}_{F^{\emptyset , \beta}_{\star ,0}/QG})} 
= \frac{[F_{\beta}]^{\vir} }{e^{\CC ^*\times T}(N^{\vir}_{F_{\beta}/Q_{ \PPb} } ) }      . \] 
Here again the localization residues are taken as sums over the 
connected components
of $F^{\emptyset, \beta}_{\star , 0}$ and of $F_\beta$.
\end{enumerate}
\end{Lemma}
\begin{proof}
(1): Take the open neighborhood in $QG$ by imposing the condition that
 the domain curves are irreducible and the open neighborhood in  $Q_{ \PPb}$ by requiring that
 the base points are away from the stacky point $[1,0]\subset \PP _{a, 1}$.

(2): We compare the $\CC^*$ moving and fixed parts of both obstruction theories.
First for $QG$, we need to look at the fixed part of $R^\bullet \pi _* ([u]_1^*T_{\X} \boxtimes [u]_2^* T\PP ^1)$ and the fixed part of
the infinitesimal automorphism/deformation of $(C, x_\star)$. Altogether its contribution coincides with the fixed part of $R^\bullet [u]^*T_{\X}$.
The Euler class of the moving part of them altogether becomes the Euler class of the moving part of $R^\bullet [u]^*T_{\X}$ 
divided by $(-z)$.
\end{proof}

By Lemma \ref{J_form} and Lemma \ref{Nbhd}, we obtain the following.

\begin{Prop}\label{I}
\[ I(0, q, z) = \one _X + \sum _{\beta \ne 0} q^\beta (\tilde{ev}_\star)_* \frac{[F_{\beta}]^{\vir} }{e^{\CC ^*\times T}(N^{\vir}_{F_{\beta}/Q_{ \PPb} } ) } .\]
\end{Prop}

\subsection{$\mathds{I}$-function for twisted theory}\label{twisted big_I} Let $\cE$ be a convex vector bundle on $\X$ as in Remark \ref{twisted}. 
Let $J^{0+,\cE}(t,q,z)$ be the $J$-function of the $(\cE, Euler)$ twisted $(0+)$-quasimap theory and let 
$$I^{\cE}(q,z)=\sum q^\beta I_\beta^{\cE}(q,z)$$ be its specialization at
$t=0$. The proof of Theorem \ref{big_I} applies for the twisted theory as well and we conclude the following.
\begin{Thm} The twisted $\mathds{I}$-function
\begin{equation*} \mathds{I}^{\cE} (t, q, z) : = \sum  _{\beta} q^{\beta} \exp (\frac{1}{z}\sum _i t_i p_i( c_1 (L_{\eta_{ij} })+\beta(L_{\eta_{ij} })z )  
I_\beta^{\cE}  (q, z). \end{equation*}
is on the Lagrangian cone of the $(\cE,Euler)$-twisted Gromov-Witten theory of $X$.
\end{Thm}

By convexity, $R^0\pi_*[u]^*\cE$ is a vector bundle on both $QG = QG ^{0+} _{0, \star, \beta } (X )$ and $Q_{ \PPb} = Q_{\PPb}(X, \beta)$ for each $\beta\neq 0$. The restrictions of the two vector bundles are identified by the natural isomorphism  between $F^{\emptyset , \beta}_{\star ,0}$ and $F_{\beta} $ from Lemma \ref{Nbhd}. From this and Lemma \ref{Nbhd}(2), it follows that
\begin{equation} I_\beta^{\cE}  (q, z)=(\tilde{ev}_\star)_* \frac{e^{\CC ^*\times T\times\mathbb{G}_m}(R^0\pi_*[u]^*\cE|_{F_\beta})\cap[F_{\beta}]^{\vir} }{e^{\CC ^*\times T}(N^{\vir}_{F_{\beta}/Q_{ \PPb} } ) }.
\end{equation}

\section{Toric Deligne-Mumford stacks} In this section we make Theorem \ref{big_I} completely explicit for toric DM stacks 
by calculating the localization residues $I_\beta(0,q,z)$ via stacky loop spaces. For toric
manifolds this is a well known calculation with Euler sequences, due to Givental \cite{Giv}. For the convenience of the reader we present its extension to the orbifold case.
As a result of this calculation, the Mirror Theorem for toric DM stacks, recently proved by different methods in \cite{CCIT}, becomes a special case of Theorem \ref{big_I}.


\subsection{Set-up}

Let $G$ be the algebraic torus $(\CC ^*)^r$, $r\ge 0$.
Fix a finite collection $[N]$ of (not necessarily distinct) characters of $G$. 
In this section, we consider the case
\[ W = \bigoplus _{\rho \in [N]} \CC _\rho , \ \ G = (\CC ^*)^r .\]
Fix a character $\theta$ of $G$ and assume that $W^{ss}: = W^{ss}(\theta)=W^{s}(\theta)$ as before.
For a character $\rho$, the associated line bundle on $\X:=[W/G]$ will be denoted by $L_\rho$.
Let $\pi _i$ be the $i$-th \lq\lq standard" character $$\pi_i:G=(\CC ^*)^r\lra \CC ^*,$$ coming from the $i$-th projection.
As the character group of $G$ is the free abelian group generated by $\pi_i$,
there are unique integers $a_{i,\rho}$ making $\rho = \sum _i a_{i, \rho }\pi _i$.

The GIT stack quotient $X:=[W^{ss}/G]$ is a toric DM stack (in the sense of \cite{BCS}), with quasi-projective coarse moduli space. Since it is a global quotient stack by an abelian group,
it is known that 
\begin{align*}  \IX  = \coprod _{g \in G} [(W^{ss})^g/G] ,   \text{ and hence }  
   \bIX   =   \coprod _{g\in G} [(W^{ss})^g / (G/\lan g \ran )] .\end{align*}
Let $T=(\CC ^*)^{[N]}$ be the big torus with the standard action on $W$.
There are only finitely many fixed points and finitely many 1-dimensional orbits on $\uX$ under the induced $T$ action.
Therefore Theorems \ref{Comp_Thm} and \ref{big_I} apply to the triple $(W,G,\theta)$.

In what follows, we will denote $D_\rho$  the hyperplane of $W$  associated to $\rho$
as well as the corresponding $T$-equivariant divisor classes of $[W/H]$ (or even its restriction to various substacks of $[W/H]$),
whenever $H$ is an algebraic group acting on $W$ with the hyperplane $D_\rho$ being invariant.

\subsection{Explicit description of $\Hom _\beta (\PPa, [W/G])$}\label{Toric_Comp}

Let $\mathcal{S}ch_{\CC}$ be the category of schemes over $\Spec\CC$ and let $U:= \CC ^2\setminus \{ 0 \}$.

Since $\Pic U \cong \Pic \CC ^2$ is trivial, any line bundle on $\PPa$ is obtained from the Borel's 
mixed construction $\mathcal{O}_{\PPa}(m):= [U\times \CC _m / \CC ^*]$ for some weight $m\in \ZZ$ of $\CC ^*$.
The degree of $\mathcal{O}_{\PPa } (m)$ is $m/a \in \QQ$.

We find an explicit description of Hom-stack $\Hom _{\beta} (\PPa, \X)$.
The groupoid fiber over $S\in \mathcal{S}ch_{\CC}$ of the category $\Hom _{\beta} (\PPa, \X)$ is equivalent to
the category described as follows:
\begin{itemize}
\item Objects are collections $$(\cL_i, i\in [r]; u_\rho, \rho \in [N])$$ 
consisting of $\CC ^*$-equivariant line bundles $\cL_i$ on 
$S\times U$,  $i=1,\dots,r$,  together with  $\CC ^*$-invariant sections $u_\rho$ of $$\cL _\rho := \ot _ i \cL _i ^{\ot a_{i\rho}}.$$
By triviality of $\Pic U$, the $\CC ^*$-equivariant line bundle $\cL_i$
is determined uniquely by an integer weight of $\CC ^*$ and a line bundle $M_i$ on $S$. For the collection to give a map of class $\beta$, the weight must be  $a\beta(L_{\pi _i}) \in \ZZ$:
$$\cL _i = M_i \boxtimes (U\times\CC_{a\beta (L_{\pi _i} )}).$$
\item
Arrows from $(\cL_i, i\in [r]; u_\rho, \rho \in [N])$  to $(\cL'_i, i\in [r]; u'_\rho, \rho \in [N])$ are collections of
isomorphisms $(\varphi _i\in \Gamma (S\times U, \cL _i^{\vee}\ot \cL _i')^{\CC^*}, i\in [r])$ for which $u_\rho$ corresponds to $u'_\rho$ for every $\rho$.
\end{itemize}

Consider the graded ring $\CC[x, y]$ with $\deg x= a$, $\deg y = 1$ and
denote by $\CC[x, y]_m$ its $\CC$-subspace of degree $m$.
Pushing-forward along $U$, together with the fact that $H^0(U, \mathcal{O}_U) = H^0(\CC ^2, \mathcal{O}_{\CC ^2}) = \CC [x, y]$,
we may regard  $u_\rho$ canonically as an element  of $$ \Gamma (S, M_\rho) \ot  _{\CC}\CC [x, y] _{a\beta(L_\rho)}, $$ 
where $M_\rho := \ot _i M_i ^{\ot a_{i,\rho}}$.
 
Hence the groupoid fiber above is equivalent to
the category whose objects are  $(M_i, u_\rho\in \Gamma (S, M_\rho) \ot  _{\CC}\CC [x, y] _{a\beta(L_\rho)} , \rho \in [N])$ and whose arrows from
$(M_i, u_\rho\in \Gamma (S, M_\rho) \ot _{\CC} \CC [x, y] _{a\beta(L_\rho)} , \rho \in [N])$ to $(M_i', u'_\rho\in \Gamma (S, M'_i)\ot \CC [x, y] _{a\beta(L_\rho)} , \rho \in [N])$
are collections of $\cO _S$-module isomorphisms  $M_i\ra M_i'$, $i\in[r]$, compatible with $u_\rho$, $u'_\rho$.

Given $a$ and $\beta$, consider the finite dimensional vector space \[ W_\beta^a := \bigoplus _{\rho \in [N]} \CC [x, y] _{a\beta (L_\rho )} \] 
with the $G$ action given by the direct sum of the diagonal $G$ action on $\CC[x,y ]_{a\beta (L_\rho)}$ by the weight $\rho$, so that $\CC[x,y ]_{a\beta (L_\rho)}\cong \bigoplus  \CC _\rho$.
In particular, when $\beta =0$, we recover $W_0 = W$ with the original $G$ action. Now the conclusion of the equivalent descriptions 
 of the groupoids above can be stated as follows.

\begin{Prop}\label{BigToric}
The following stacks are all equivalent:
\[ \Hom _{\beta} (\PPa, \X) \cong \Hom _{\beta}([\CC ^2 / \CC ^*], \X) \cong [W_\beta^a / G] . \] 
\end{Prop}

\medskip

\noindent{\bf Convention}: From now on we let  $a$ be the minimal positive integer associated to $\beta$ by Lemma \ref{minimal} and will write $W_\beta$ for $W_\beta^a$.

\medskip

\begin{Cor} \label{explicit}
The stack $Q_{\PPa}(X, \beta)$ 
is equivalent to the smooth quotient stack $[W_\beta ^{ss}(\theta) /G]$. 
\end{Cor}
\begin{proof}
By Proposition \ref{BigToric} and Lemma \ref{minimal}, $[u]\in  \Hom _{\beta} (\PPa, \X)$ is a $\CC$-point of $Q_{\PPa}(X, \beta)$ if and only if $u_\rho (\zeta _0, \zeta _1) \in W^{ss}(\theta)$
for general points $(\zeta _0, \zeta _1)\in\CC^2\setminus\{(0,0)\}$.  It is easy to see that the latter condition is equivalent to $(u_\rho)_{\rho} \in W^{ss}_\beta(\theta)$.
\end{proof}

Recall  the perfect obstruction theory $R^\bullet \pi _*[u]^*  \mathbb{T}_{[W/G]}$ for $Q_{\PPa}(X,\beta )$ as in \S \ref{Stacky_Loop}.
We can describe the complex more explicitly since $[W^{ss}_\beta/G]$ is smooth.  
Consider the generalized Euler sequence (see \cite[\S 5.1]{CKM}), i.e.,  the distinguished triangle
\[ \mathfrak{g}\times _G W \ra W\times _G W \ra \mathbb{T}_{[W/G]} \] 
 on $[W/G]$.  
It induces
an exact sequence on $[W^{ss}_\beta /G]$
\[ 0\ra \mathcal{O}^{\oplus [r]} \ra \oplus _{\rho \in [N]}\pi _* [u]^* L _\rho \ra \mathbb{T}_{[W^{ss}_\beta/G]} \ra 0, \]
and the obstruction vector bundle $ \oplus _{\rho \in [N]}R^1\pi _*  [u]^* L _\rho$  defining the virtual fundamental
class of $[W^{ss}_\beta/G] \cong Q_{\PPa}(X,\beta )$.

\subsection{The virtual normal bundle}
Recall the $\CC^*$ action on the coarse moduli $\PP ^1$, given as $t\cdot [\zeta _0, \zeta _1] = [t\zeta _0, \zeta _1]$.
This induces an action on $[W^{ss}_\beta/G]$. 
Define  \begin{equation}
Z_\beta  := \bigoplus _{\rho \in [N], \beta (L_\rho) \in \ZZ _{\ge 0} } \CC \cdot x ^{\beta(L_\rho )}  
  \subset W_\beta  . \label{Zbeta} \end{equation}
  Under the natural identification of $Z_\beta$  with the sub-$G$-representation $ \bigoplus _{\rho \in [N], \beta (L_\rho) \in \ZZ _{\ge 0} } \CC_\rho$ of $W$, we have
  \begin{equation*} Z_\beta=\bigcap _{\rho : \beta(L_\rho) <0 \text{ or } \beta (L_\rho) \not\in \NN} D_{\rho}.
  \end{equation*}
  
  The  $\CC^*$-fixed component $[Z_\beta / G]$ of $\Homa$ is distinguished in the sense that
$$F_\beta\cong [Z_\beta / G] \cap [W^{ss}_\beta /G] $$ under the identification from Corollary \ref{explicit}. Let $Z_\beta ^{ss} := Z_\beta \cap W^{ss}_\beta$.

\begin{Prop}\label{VirtNormal}
\begin{align*} & e^{\CC^*\times T} (N_{F_{\beta}/QG_{\PPa}}^{\mathrm{vir}}) =    e^{\CC^*\times T} (N_{[Z_\beta^{ss}/G] /[W^{ss}_{\beta}/G]})  \\
& =\frac{\prod_{\rho: \beta(L_\rho)  >0}\prod_{0\leq \nu < \beta(L_\rho) }(D_\rho+( \beta(L_\rho) -\nu)z)}
{\prod_{\rho: \beta(L_\rho) <0}\prod_{\lfloor \beta(L_\rho)+1\rfloor\leq \nu <0}(D_\rho+( \beta(L_\rho) -\nu)z)} \end{align*}
 and $[F_{\beta} ]^{\vir} = [F_{\beta}] = [Z^{ss}_\beta /G]$. (In the above formula the index $\nu$ in the products runs over integers and $\lfloor\;\rfloor$ denotes the round-down of a rational number.) 
\end{Prop}
\begin{proof}

First note that  $\pi _* [u]^* L _\rho = \mathcal{O}(D_\rho ) \ot H^0(\PP _{a,1}, \mathcal{O}(a\beta(L_\rho))) $. 
The $\CC^*$-representation space  $H^0(\PP _{a,1}, \mathcal{O}(a\beta(L_\rho)))$ (for $\beta(L_\rho)\ge 0)$ has the basis 
$$y^{a\beta (L_\rho )}, y^{a(\beta (L_\rho)-1)} x, ..., y^{a (\beta (L_\rho ) - \lfloor \beta (L_\rho ) \rfloor ) } x^{\lfloor \beta (L_\rho ) \rfloor  }, $$
therefore its $\CC^*$ weights 
are $\beta (L_\rho)$, $\beta (L_\rho) -1$, ..., $\beta (L_\rho) - \lfloor \beta (L_\rho )\rfloor$.
This explains the numerator of $e^{\CC^*\times T} (N_{F_{\beta}/QG_{\PPa}}^{\mathrm{vir}})$.

The obstruction bundle  $\bigoplus _\rho R^1\pi _* [u]^*L _\rho$ becomes  $$\bigoplus _{ \rho : \beta (L_\rho)\in \QQ_{< 0}} \mathcal{O}(D_\rho ) \ot H^1 (\PP _{a, 1}, \mathcal{O}(a\beta (L_\rho) ).$$
A $\CC^*$-eigenbasis of the cohomology space $H^1 (\PP _{a, 1}, \mathcal{O}(a\beta (L_\rho) )$ can be computed by taking \v Cech 1-cocycles with respect to the \'etale covering of $\PPa$:
$$  y^{a(\beta (L_\rho ) + 1)}x^{-1}, y^{a(\beta (L_\rho) + 2)} x^{-2}, ..., y^{a (\beta (L_\rho ) - \lfloor \beta (L_\rho ) +1  \rfloor ) } x^{\lfloor \beta (L_\rho ) + 1 \rfloor  } . $$
Therefore the weights are $\beta (L_\rho) - \nu$ with $\lfloor \beta(L_\rho)+1\rfloor\leq \nu <0$, $\nu \in \ZZ$. This is the denominator of $e^{\CC^*\times T} (N_{F_{\beta}/QG_{\PPa}}^{\mathrm{vir}})$.

Finally, the obstruction bundle defining the virtual fundamental class of the smooth stack $[W_\beta ^{ss}(\theta )/G]=[W^{s}_\beta (\theta )/G]$,
when restricted to $F_{\beta}$ has no $\CC ^*$-fixed part. We conclude that  $[F_{\beta} ]^{\vir} = [F_{\beta}]$.
\end{proof}

Let $g_\beta = (e^{2\pi \sqrt{-1}\beta (L_\rho)}) _\rho \in G$.
By the proof of Lemma \ref{minimal}, note that the connected component into which $F_{\beta}$ lands under the evaluation
map at $\star$ and its associated component in the unrigidified cyclotomic inertia stack are respectively
\begin{align} 
  &  [(W^{ss} \cap \bigcap _{\rho : \beta(\rho) \not\in  \ZZ}  D_\rho ) / (G/\lan g_\beta \ran) ]   \text{ and } \nonumber \\
  &  [(W^{ss})^{g_\beta} / G ] = [(W^{ss} \cap \bigcap _{\rho : \beta(\rho) \not\in  \ZZ}  D_\rho ) / G ] . \label{IX} \end{align}

The evaluation map $ev_\star$ is factored through  $[(W^{ss})^{g_\beta} / G ]$ and makes
$F_{\beta}= [Z^{ss}_\beta /G]$ as a closed substack of  $[(W^{ss})^{g_\beta} / G ]$, whose normal bundle yields, by \eqref{Zbeta} and \eqref{IX},
 \begin{equation}\label{Neg}
 e^{\CC^*\times T}(N_{ [Z^{ss}_\beta /G]/[(W^{ss})^{g_{\beta}} / G ] })= e^{\CC^*\times T} (N_{Z_\beta /W^{g_\beta}  }) _{|_{  [Z^{ss}_\beta /G]} }= \prod_{\rho: \beta(L_\rho) \in \ZZ_{<0}} D_\rho.
 \end{equation}

\begin{Thm}\label{I_Thm}
\[ I(0, q,z) =\sum_{\beta\in \Eff (W, G, \theta )}q^\beta 
\frac{\prod_{\rho : \beta (L_\rho) <0}\prod_{\beta(L_\rho)\leq \nu <0}(D_\rho +(\beta(L_\rho)-\nu)z)}{\prod_{\rho :\beta(L_\rho)>0}
\prod_{0\leq \nu <\beta(L_\rho)}(D_\rho +(\beta(L_\rho)-\nu)z)}  \one _{g^{-1}_\beta} \] 
where $\one _{g^{-1}_\beta}$ is the fundamental class of $[(W^{ss})^{g^{-1}_{\beta}}/ (G/\lan g^{-1}_{\beta}\ran ) ]$.
\end{Thm}
\begin{proof}
This is immediate from Proposition \ref{I}, Proposition \ref{VirtNormal}, and equation \eqref{Neg}.
The factor $\br$ disappears since the rigidification $\varpi$ has degree $\br ^{-1}$.
\end{proof}

\begin{Cor}\label{I_0=1}
For a semi-positive triple $(W=\CC ^{[N]}, G=(\CC^*)^r, \theta)$, $I(0, q, z)$ takes form
\[ \one _X + \frac{I_1(q)}{z} + O(1/z^2), \ I_1(q) \in (q\Lambda_K)H^{\le 2}_T(X,\Lambda_K). \]
\end{Cor}

\begin{proof}
Let $\beta \ne 0$ so that for some $\rho$, 
$\beta (L_\rho ) \ne 0$. Note that  $ \one _{g^{-1}_\beta}$ is the fundamental class $\one _X$ of the untwisted sector if and only 
if $\beta (L_\rho)\in \ZZ$ for all $\rho$.  In this case, the argument in the proof of Lemma 5.9.1 in \cite{CKg} works.
If $ \one _{g^{-1}_\beta}$ is the fundamental class of a twisted sector, then the {\em power} of $1/z$ appearing in $q^\beta$ becomes
\begin{align} & \nonumber  \sum _{\rho\in [N]}  \lceil\beta (L_\rho )\rceil + \# \text{ negative integers } \beta (L_\rho)  
\\ \nonumber & >  \beta (\det \mathbb{T}_{\X})  +   \# \text{ negative integers } \beta (L_\rho) \\ 
\label{ineq} &  \ge \beta (\det \mathbb{T}_{\X})  \ge 0. \end{align} 
This implies that $I(0, q, z) = \one _X + O(q)$.
The two inequalities in \eqref{ineq} cannot be equalities at the same time since $\beta\ne 0$. Therefore there is no twisted sector part 
in the $1/z$-coefficient of the small I-function $I(0,q,z)$.
\end{proof}

Let $$\{\gamma_0=\one_X,\gamma_1,\dots, \gamma_s,\gamma_{s+1},\dots, \gamma_l\}$$ be the part of the basis of Chen-Ruan cohomology corresponding to the untwisted sector $H^*_T(X)$, with
$\{\gamma_1,\dots,\gamma_s\}$ basis for $H^2_T(X)$. In the toric case, 
we may write each $\gamma_i$ as a polynomial in $T$-equivariant  first Chern classes of 
line bundles coming from characters of $G$, as in \eqref{poly_exp}. In particular, for the divisors $\gamma_i, i=1,\dots, s$ we write $\gamma_i=c_1(L_{\eta_i})$.

\begin{Cor}\label{toric-explicit}
\begin{enumerate}
\item Let $t=\sum_{i=0}^l t_i\gamma_i=\sum_{i=0}^l t_i p_i( c_1 (L_{\eta_{ij}}))$ be a general element of $H^*_T(X)$.
Then 
\begin{align*} 
\mathds{I} (t, q, z) & = \sum  _{\beta} q^{\beta} \exp \left(\frac{1}{z}\sum _{i=0}^l t_i p_i(c_1(L_{\eta _{ij}})  + \beta (L_{\eta _{ij}} ) z ) \right) \\
& \frac{\prod_{\rho : \beta (L_\rho) <0}\prod_{\beta(L_\rho)\leq \nu <0}(D_\rho +(\beta(L_\rho)-\nu)z)}{\prod_{\rho :\beta(L_\rho)>0}
\prod_{0\leq \nu <\beta(L_\rho)}(D_\rho +(\beta(L_\rho)-\nu)z)}  \one _{g^{-1}_\beta}
 \end{align*}
 is on the Lagrangian cone of the Gromov-Witten theory of $X$.
\item  For $t= t_0\one _X + \sum _{i=1}^s t_i \gamma_i \in H^0_T(X)\oplus H^2_T(X)$,  
\begin{align*} I^{\Giv}( t, q, z) & =
\sum _{\beta \in \Eff(W, G, \theta) }    q^\beta  e^{(t_0\one _X + \sum _{i=1}^s t_i (\gamma_i+ \beta(L_{\eta_i} ) z )) /z }  \\
& \frac{\prod_{\rho : \beta (L_\rho) <0}\prod_{\beta(L_\rho)\leq \nu <0}(D_\rho +(\beta(L_\rho)-\nu)z)}{\prod_{\rho :\beta(L_\rho)>0}
\prod_{0\leq \nu <\beta(L_\rho)}(D_\rho +(\beta(L_\rho)-\nu)z)}  \one _{g^{-1}_\beta} \end{align*}
 is on the Lagrangian cone of the Gromov-Witten theory of $X$.
 \item   If  $(W=\CC ^{[N]}, G=(\CC^*)^r, \theta)$ is a semi-positive triple,
\[ J^{\infty}(t +I_1(q), q, z) = I^{\Giv} (t, q, z), \text{ for } t\in H^{\le 2}_T(X). \]
Furthermore, the so-called {\em mirror map} $t\mapsto t+I_1(q)$ is expressed via two-pointed $(0+)$-quasimap invariants as
\[ t + I_1(q) = t + \sum_{i=0}^s\sum _{\beta \ne 0} q^\beta \gamma _i \lan \gamma ^i , \one_X \ran _{0, 2, \beta}^{0+}. \]
\end{enumerate}
\end{Cor}

\begin{proof}
(1) is immediate from Theorem \ref{big_I} and Theorem \ref{I_Thm}, and (2) is a specialization of (1).

(3) follows from  (2), Theorem \ref{JS} (3b), and  Corollary \ref{I_0=1}.
\end{proof}

In particular, part (3) of the above Corollary proves Conjecture 4.3 of \cite{Ir} after the identification of  $\HA^*_T(\bIX)$ with $\HA ^*_T(\IX)$ by
the pullback $\varpi ^*$. (Of course, the interpretation of the mirror map as a generating series of quasimap invariants is new.)

Part (2) of the Corollary is precisely the main result of \cite{CCIT}. Note that the notion of \lq\lq $S$-extended $I$-function" from \cite{CCIT} corresponds in our terminology to the Givental small $I$-function for a different GIT presentation of the geometric target $X$. 

\begin{Rmk} Combining the considerations from Remark \ref{twisted} with Theorem \ref{twisted big_I} and the calculations of this section provides a different proof (not relying on Tseng's orbifold Quantum Lefschetz theorem \cite{Tseng}) of the Mirror Theorem for complete intersections of convex hypersurfaces in toric DM stacks of
Coates, Corti, Iritani, and Tseng, see \cite[Theorem 25]{CCIT2}. We leave the easy details to the reader.

\end{Rmk}

\end{document}